\documentclass[12pt,reqno]{amsart}
\usepackage{titlesec}
\titleformat{\section}[block]{\large \bfseries}{\arabic{section}}{1em}{}[]
\titlespacing{\section}{0pt}{*1.5}{*1.1}
\titleformat{\subsection}[block]{\normalsize \bfseries}{\arabic{section}.\arabic{subsection}}{1em}{}[]
\titlespacing{\subsection}{0.5cm}{*4}{*1.5}

\usepackage{geometry}
\usepackage{lipsum}
\usepackage{amsfonts}
\usepackage{graphicx}
\usepackage{epstopdf}
\usepackage{algorithm}

\usepackage[rotateright]{rotating}
\usepackage{subfigure,fancybox}
\usepackage{amscd,amstext}
\usepackage{amsmath}
\usepackage{amssymb}
\usepackage{color}
\usepackage{multirow}
\usepackage{mathrsfs}
\usepackage{diagbox}
\usepackage{bbm}
\numberwithin{figure}{section}
\numberwithin{table}{section}

\ifpdf
  \DeclareGraphicsExtensions{.eps,.pdf,.png,.jpg}
\else
  \DeclareGraphicsExtensions{.eps}
\fi

\theoremstyle{plain}
\newtheorem{theorem}{Theorem}[section]

\newtheorem{lemma}{Lemma}[section]

\newtheorem{remark}{Remark}[section]

\usepackage{amsopn}

\numberwithin{equation}{section}

\usepackage[mathlines]{lineno}
\usepackage{etoolbox} 

\newcommand*\linenomathpatchAMS[1]{%
  \expandafter\pretocmd\csname #1\endcsname {\linenomathAMS}{}{}%
  \expandafter\pretocmd\csname #1*\endcsname{\linenomathAMS}{}{}%
  \expandafter\apptocmd\csname end#1\endcsname {\endlinenomath}{}{}%
  \expandafter\apptocmd\csname end#1*\endcsname{\endlinenomath}{}{}%
}

\expandafter\ifx\linenomath\linenomathWithnumbers
  \let\linenomathAMS\linenomathWithnumbers
  \patchcmd\linenomathAMS{\advance\postdisplaypenalty\linenopenalty}{}{}{}
\else
  \let\linenomathAMS\linenomathNonumbers
\fi

\linenomathpatchAMS{gather}
\linenomathpatchAMS{multline}
\linenomathpatchAMS{align}
\linenomathpatchAMS{alignat}
\linenomathpatchAMS{flalign}

\begin{document}

\title[Pressure-robustness in Stokes-Darcy Problem]{Pressure-robustness in Stokes-Darcy Optimal Control Problem with reconstruction operator}

\author{Jingshi Li}
\address{School of Mathematics and Statistics, Nanjing University of Information Science \& Technology, Nanjing 210044, China.}
\email{lijs@nuist.edu.cn}

\author{Jiachuan Zhang$^{\ast}$}
\thanks{*Corresponding author}
\address{School of Physical and Mathematical Sciences, Nanjing Tech University, Nanjing 211816, China.}
\email{zhangjc@njtech.edu.cn}

\author{Ran Zhang}
\address{School of Mathematics, Jilin University, Changchun 130012, Jilin, China.}
\email{zhangran@jlu.edu.cn}

\begin{abstract}

This paper presents a pressure-robust discretizations, specifically within the context of optimal control problems for the Stokes-Darcy system. The study meticulously revisits the formulation of the divergence constraint and the enforcement of normal continuity at interfaces, within the framework of the mixed finite element method (FEM). The methodology involves the strategic deployment of a reconstruction operator, which is adeptly applied to both the constraint equations and the cost functional. This is complemented by a judicious selection of finite element spaces that are tailored for approximation and reconstruction purposes. The synergy of these methodological choices leads to the realization of a discretization scheme that is pressure-robust, thereby enhancing the robustness and reliability of numerical simulations in computational mathematics.
\medskip

\noindent{\bf Keywords:}~~ pressure-robustness, PDE-constrained optimization , Stokes-Darcy problem, reconstruction operator, error estimate

\noindent{\bf \text{Mathematics Subject Classification} :}~~65N15, 65N30, 76D07.

\end{abstract}

\maketitle

\section{Introduction}\label{sec:introduction}
This paper is dedicated to the investigation of a pressure-robust FEM tailored for optimal control problems within the context of the coupled Stokes-Darcy system. This system is a cornerstone in various fields, including energy engineering, battery technology, biomedicine, and hydrogeology, among others.
Traditional discretization techniques often involve a relaxation of the divergence constraint and the interface normal continuity, opting instead for a discrete enforcement of these conditions. However, such relaxations can introduce pressure-dependent consistency errors that may adversely affect the accuracy of the computed velocity field.
Our analysis reveals that the velocity error can be polluted by these pressure-related consistency errors, which are quantified by the dual norm,
\begin{align*}
\sup_{\boldsymbol{\psi}_h\in V_h(0)}\frac{(\nabla\cdot\boldsymbol{\psi}_h,p)
-\langle\boldsymbol{\psi}_h^s\cdot\boldsymbol{n}^s
+\boldsymbol{\psi}_h^d\cdot\boldsymbol{n}^d,p^d\rangle_{\Gamma}}
{\|\boldsymbol{\psi}_h\|_X}.
\end{align*}
The discrete function space $V_h(0)$ is characterized by a collection of functions that, in a discrete context, adhere to the principles of divergence-free and interface normal continuity.
 If these properties are satisfied exactly, as denoted by the conditions $\nabla\cdot\boldsymbol{\psi}_h=0$ and $\boldsymbol{\psi}_h^s\cdot\boldsymbol{n}^s
+\boldsymbol{\psi}_h^d\cdot\boldsymbol{n}^d=0$,  the consistency error associated with the discretization is nullified. 
Should these conditions not be met precisely, the consistency error becomes contingent upon the pressure level and is challenging to mitigate, particularly in scenarios where the pressure dynamics are dominant.
Addressing the pressure dependency introduced by the discrete divergence-free constraint, the literature presents three distinct methodologies to remove the dependency: divergence-free FEM constructed by de Rham compelex (or Stokes complex) \cite{falkStokesComplexesConstruction2013,guzmanConformingDivergencefreeStokes2014},  modified $H(div)$-conforming FEM with techniques in discontinuous Galerkin methods \cite{cockburnNoteDiscontinuousGalerkin2007,wangNewFiniteElement2007}, and modified mixed FEM with divergence-free reconstruction operator for the test functions \cite{Linke2016,ledererDivergencefreeReconstructionOperators2017}. A review of the three approaches can be found in \cite{johnDivergenceConstraintMixed2017}. 


In this work, we engage the reconstruction operator, denoted by 
$\Pi_h$, as a pivotal tool to eradicate the dependency on pressure within our computational framework. To fortify the robustness of our optimal control problems against pressure variations, we implement a dual strategy: firstly, by applying the reconstruction operator $\Pi_h$ to the test function situated on the right-hand side of the constraint equation, and secondly, by recalibrating the cost functional to the following form:

\begin{align*}
\min \frac{1}{2}\|\boldsymbol{u}_h-\boldsymbol{u}^*\|^2+\frac{\alpha}{2}
\|\boldsymbol{\eta}_h\|^2\quad\longrightarrow\quad \min \frac{1}{2}\|\Pi_h\boldsymbol{u}_h-\boldsymbol{u}^*\|^2+\frac{\alpha}{2}
\|\boldsymbol{\eta}_h\|^2.
\end{align*}
These methodological refinements, previously established for achieving pressure robustness in the context of Stokes optimal control problems \cite{merdonPressureRobustnessContextOptimal2023}, are now adapted to the Stokes-Darcy optimal control problem. The dependency on pressure is consequently transformed, as articulated in the modified consistency error expression:
\begin{align}\label{eqn:modified consistency error}
\sup_{\boldsymbol{\psi}_h\in V_h(0)}\frac{(\nabla\cdot(\Pi_h\boldsymbol{\psi}_h),p)
-\langle\Pi_h^s\boldsymbol{\psi}_h^s\cdot\boldsymbol{n}^s
+\Pi_h^d\boldsymbol{\psi}_h^d\cdot\boldsymbol{n}^d,p^d\rangle_{\Gamma}}
{\|\boldsymbol{\psi}_h\|_X}.
\end{align}

By selecting compatible finite element spaces for both the approximation and reconstruction spaces within the fluid domain and the saturated porous medium, the reconstruction operator can effectively mitigate the pressure dependency associated with both the discrete divergence-free condition and the discrete interface normal continuity. This is articulated through the following implications:
\begin{align*}
(\nabla\cdot\boldsymbol{\psi}_h, \phi_h)=0,~\forall \phi_h\in Q_h &\Longrightarrow \nabla\cdot\Pi_h\boldsymbol{\psi}_h=0,\\
\langle\boldsymbol{\psi}_h^s\cdot\boldsymbol{n}^s
+\boldsymbol{\psi}_h^d\cdot\boldsymbol{n}^d,\lambda_h\rangle_{\Gamma}=0,~\forall
\lambda_h\in\Lambda_h 
&\Longrightarrow \Pi_h^s\boldsymbol{\psi}_h^s\cdot\boldsymbol{n}^s
+\Pi_h^d\boldsymbol{\psi}_h^d\cdot\boldsymbol{n}^d=0.
\end{align*}
Consequently, the consistency error delineated in the modified consistency error expression evaporates, as substantiated by Lemma~\ref{lem:Pi}. A concrete instantiation of such spaces is delineated by the selection of Raviart-Thomas (RT) elements of identical order for the approximation and reconstruction spaces in the Darcy region, as well as for the reconstruction space in the Stokes region, with the approximation space in the Stokes region being the classical mixed finite element space. Through the synergistic application of the reconstruction operator and the judicious assignment of finite element spaces, we secure the pressure robustness for the Stokes-Darcy optimal control problem.


The subsequent sections of this paper are meticulously organized as follows: Section 2 presents a comprehensive introduction to the Stokes-Darcy problem, delineating its variational formulation and the associated optimal control problem. Section 3 delves into the intricacies of the classical FEM discretization, scrutinizing its pressure dependency and its consequential impact on the velocity error. Section 4 introduces a novel modification to the Stokes-Darcy optimal control problem through the integration of a reconstruction operator, which paves the way for a pressure-robust discretization scheme. This structured approach ensures a coherent progression of topics, facilitating a clear understanding of the theoretical underpinnings, numerical challenges, and the proposed solutions within the realm of computational mathematics.

\section{Problem statements}
This paper concerned with pressure robustness finite element methods of the optimal control problem 
subject to the Stokes-Darcy problem 
\begin{align}\label{eqn:state Stokes}
&-2\mu\nabla\cdot D(\boldsymbol{u}^s)+\nabla p^s=\boldsymbol{f}^s \mbox{~in~} \Omega^s,\quad
\nabla\cdot \boldsymbol{u}^s=0 \mbox{~in~} \Omega^s,\quad
\boldsymbol{u}^s=0 \mbox{~on~} \Gamma^s,\\
\label{eqn:state Darcy}
&\mu K^{-1}\boldsymbol{u}^d+\nabla p^d=0 \mbox{~in~} \Omega^d,\quad
\nabla\cdot \boldsymbol{u}^d=g^d \mbox{~in~} \Omega^d,\quad
\boldsymbol{u}^d\cdot \boldsymbol{n}^d=0 \mbox{~on~} \Gamma^d.
\end{align}
with the following interface conditions on $\Gamma$
\begin{align}\label{eqn:state boundary 1}
&\boldsymbol{u}^s\cdot\boldsymbol{n}^s+\boldsymbol{u}^d\cdot\boldsymbol{n}^d=0,\\
\label{eqn:state boundary 2}
&p^s-2\mu D(\boldsymbol{u}^s)\boldsymbol{n}^s\cdot\boldsymbol{n}^s=p^d,\\
\label{eqn:state boundary 3}
&\boldsymbol{u}^s\cdot\boldsymbol{\tau}_j=-2\frac{\sqrt{\kappa_j}}{\alpha_1}D(\boldsymbol{u}^s)
\boldsymbol{n}^s\cdot\boldsymbol{\tau}_j, \quad j=1,\cdots,N-1.
\end{align}

In the above equations, $\Omega^i, (i=s,d)$ are the bounded and simply connected domain in $\mathbb{R}^N, (N=2, 3)$, which are free fluid region and the Darcy flow, respectively. Let the interface $\Gamma=\Omega^s\cap \Omega^d$, and $\Gamma^i=\partial\Omega^i\backslash\Gamma$. All interfaces and boundaries are hypothesized to be the polygonal or polyhedral in $\mathbb{R}^N$. Also let $\boldsymbol{n}_i$, be the unit normal vectors outward to $\partial\Omega^i$. Specifically, $\boldsymbol{n}^s=-\boldsymbol{n}^d$ on $\Gamma$. The fluid velocity $\boldsymbol{u}^i$ and fluid pressure $p^i$ play different mathematical roles in different regions. In the free fluid domain $\Omega^s$, 
$(\boldsymbol{u}^s,p^s)$ satisfies the Stokes equations. And in the porous media domain
$\Omega^d$, $(\boldsymbol{u}^d,p^d)$ satisfies the Darcy flow equations.
Here, $\boldsymbol{f}^s$ and $g^d$ describe the body forces and divergence constraint with solvability condition $\int_{\Omega^d}g^d=0$,
$\mu$ is the fluid viscosity, 
$D(\boldsymbol{u}^s)=(\nabla\boldsymbol{u}^s+(\nabla\boldsymbol{u}^s)^T)/2$,
$\boldsymbol{\tau}_j$ denote the unit
tangent vectors to $\Gamma$, 
$\alpha_1$ is an empirical parameter obtained through experiments,
$K_j$ is the 
symmetric and positive definite permeability tensor in $\Omega_2$, which satisfying 
\begin{align*}
K_L\xi^T\xi\leq \xi^TK_j\xi\leq K_U\xi^T\xi, \quad \forall \xi\in\mathbb{R}^N,
\end{align*} 
for some constants $0< K_L\leq K_U<\infty$, and
$\kappa_j=\boldsymbol{\tau}_j\cdot K_j\cdot\boldsymbol{\tau}_j$. 
 Moreover, (\ref{eqn:state boundary 1}) represents normal continuity across $\Gamma$, balance of normal forces is expressed by (\ref{eqn:state boundary 2}), and (\ref{eqn:state boundary 3}) is the famous Beavers–Joseph–Saffman conditions. 
\subsection{Weak formulation}
Before deriving the weak formulation, let us first introduce some notation for the Sobolev space. For a subdomain $E\subset\mathbb{R}^N$ and $m\geq 0$, we denote by  $\|\cdot\|_{m,E}$ and  $|\cdot|_{m,E}$ the standard norm and seminorm of the usual Sobolev space $H^m(E)$. In particular, $H^0(E)=L^2(E)$ and the subscript $m$ will be dropped from the norms. The norms are applicable to both vector valued functions in space $[H^m(E)]^N$ and tensor valued functions in space $[H^m(E)]^{N\times N}$ as well.
Let $(\cdot,\cdot)_E$ denote the $L^2(E),[L^2(E)]^N$, and $[L^2(E)]^{N\times N}$ inner product for scalar, vector, and tensor valued functions, respectively. For example, for any vector valued functions $\boldsymbol{\psi},\boldsymbol{\phi}\in [L^2(E)]^N$, tensor valued functions $A,B\in [L^2(E)]^{N\times N}$, we have
\begin{align*}
(\boldsymbol{\psi},\boldsymbol{\phi})_E=\sum_{i=1}^N\int_E\psi_i\phi_i,
\quad (A,B)_E=\sum_{i,j=1}^N\int_EA_{ij}B_{ij}=\int_EA:B.
\end{align*}
In the above definition, if $E=\Omega^i, i=s,d$, we will abbreviate as $\|\cdot\|_{k,i}$,
$|\cdot|_{k,i}$, and $(\cdot,\cdot)_i$. We also use $\langle\cdot,\cdot\rangle_{\Gamma}$ and $\|\cdot\|_{\Gamma}$ to denote the $L^2(\Gamma)$ inner product and norm, respectively, for scalar and vector valued functions.
Another scale-valued and another vector-valued function spaces also need to defined as follows:
\begin{align*}
L_0^2(E)&=\{q\in L^2(E)~|~\int_E q=0\},\\
H(div;E)&=\{\boldsymbol{v}\in [L^2(E)]^N~|~\nabla\cdot\boldsymbol{v}\in L^2(E)\}.
\end{align*}

Now, we are ready to define the function spaces for both Stokes region and Darcy region. Let
\begin{align*}
V^s=\{\boldsymbol{v}^s\in [H^1(\Omega^s)]^N~|~\boldsymbol{v}^s=0 \mbox{~on~} \Gamma^s\},
\end{align*}
and
\begin{align*}
V^d=\{\boldsymbol{v}^d\in H(div;\Omega^d)~|~\boldsymbol{v}^d\cdot\boldsymbol{n}^d=0 \mbox{~on~} \Gamma^d\}.
\end{align*}
with the norm
\begin{align*}
\|\boldsymbol{v}\|_X=(|\boldsymbol{v}^s|_{1,s}^2+\|\boldsymbol{v}^d\|_d^2
+\|\nabla\cdot\boldsymbol{v}^d\|_d^2)^{1/2},
\end{align*}
for any $\boldsymbol{v}=(\boldsymbol{v}^s,\boldsymbol{v}^d)\in V^s\times V^d$.
Based on the two spaces, the space for the velocity can be defined as 
\begin{align*}
V=\{\boldsymbol{v}=(\boldsymbol{v}^s,\boldsymbol{v}^d)\in V^s\times V^d~|~
\langle\boldsymbol{v}^s\cdot\boldsymbol{n}^s+\boldsymbol{v}^d\cdot\boldsymbol{n}^d,\lambda\rangle_{\Gamma}=0, \forall\lambda\in\Lambda\}.
\end{align*}
where $\Lambda=H_{00}^{1/2}(\Gamma)\subset L^2(\Gamma)$ is the completion of the smooth functions with compact support in $\Gamma$ with respect to a well-defined norm (refer to \cite{Layton2002}).
And we define the pressure space $Q=L_0^2(\Omega)$ for $\Omega=\Omega^s\cup\Omega^d$.
Then, with $\boldsymbol{f}=(\boldsymbol{f}^s,0)\in [L^2(\Omega^s)]^N\times [L^2(\Omega^d)]^N$ and 
$g=(0,g^d)\in L^2(\Omega^s)\times L^2(\Omega^d)$, 
we have the weak formulation of (\ref{eqn:state Stokes})-(\ref{eqn:state boundary 3}): find $(\boldsymbol{u},p)\in V\times Q$ such that
\begin{align}\label{eqn:state weak 1}
a(\boldsymbol{u},\boldsymbol{v})+b(\boldsymbol{v},p)&=(\boldsymbol{f},
\boldsymbol{v}),
\quad \forall\boldsymbol{v}\in V,\\
\label{eqn:state weak 2}
b(\boldsymbol{u},q)&=(g,q), \quad \forall q\in Q,
\end{align}
where
\begin{align*}
&a(\boldsymbol{u},\boldsymbol{v})\\
=&a_s(\boldsymbol{u}^s,\boldsymbol{v}^s)
+a_d(\boldsymbol{u}^d,\boldsymbol{v}^d)
+a_I(\boldsymbol{u}^s,\boldsymbol{v}^s)\\
=&2\mu(D(\boldsymbol{u}^s),D(\boldsymbol{v}^s))_s+\mu(K^{-1}\boldsymbol{u}^d,\boldsymbol{v}^d)_d
+
\sum_{j=1}^{N-1}\frac{\alpha_1\mu}{\sqrt{\kappa_j}}\langle\boldsymbol{u}^s\cdot\boldsymbol{\tau}_j,
\boldsymbol{v}^s\cdot\boldsymbol{\tau}_j\rangle_{\Gamma},\quad \forall \boldsymbol{u},\boldsymbol{v}\in V^s\times V^d,\\
&b(\boldsymbol{v},q)=-(\nabla\cdot\boldsymbol{v}^s,q)_s-(\nabla\cdot\boldsymbol{v}^d,q)_d,
\quad \forall \boldsymbol{v}\in V^s\times V^d, q\in Q.
\end{align*}
Note that the condition (\ref{eqn:state boundary 1}) are included in the spatial construction of $V$, and (\ref{eqn:state boundary 2}) and (\ref{eqn:state boundary 3}) are implicit in the derivation of the formulation (\ref{eqn:state weak 1}). The solvability of problem (\ref{eqn:state weak 1}) and (\ref{eqn:state weak 2})
can be found in Theorem~3.1 in \cite{Layton2002}.

\subsection{Optimal control problem}
Consider the optimal control problem: for given data $\boldsymbol{f}\in W=[L^2(\Omega)]^N$,
seek state and control
$(\boldsymbol{u},\boldsymbol{\eta})\in V(g)\times W$, with $V(g)=\{\boldsymbol{v}\in V~|~\nabla\cdot\boldsymbol{v}^s=0,~\nabla\cdot\boldsymbol{v}^d=g^d\}$, solving
\begin{align}
\label{eqn:optimal control problem P}
\begin{aligned}
\min_{(\boldsymbol{u},\boldsymbol{\eta})\in V(g)\times W}\frac{1}{2}\|\boldsymbol{u}-\boldsymbol{u}^*\|^2
+\frac{\alpha}{2}\|\boldsymbol{\eta}\|^2,&\\
 \mbox{s.t.~} a(\boldsymbol{u},\boldsymbol{v})=(\boldsymbol{f}+\boldsymbol{\eta},\boldsymbol{v}),&
\end{aligned}
\end{align}
for any $\boldsymbol{v}\in V(g)$. Here, $\alpha$ is a positive constant, $\boldsymbol{u}^*$ are the target vectors, and $\boldsymbol{\eta}$ is the control vectors. 
Since this is a linear quadratic optimization problem, standard theory, e.g., \cite{Reyes2015}, gives the necessary and sufficient optimality conditions with an adjoint state $\boldsymbol{w}\in V(g)$ satisfying:
\begin{align*}
a(\boldsymbol{u},\boldsymbol{\psi})-(\boldsymbol{f}+\boldsymbol{\eta},\boldsymbol{\psi})&=0,
\quad \forall \boldsymbol{\psi}\in V(g),\\
a(\boldsymbol{\psi},\boldsymbol{w})-(\boldsymbol{u}-\boldsymbol{u}^*,\boldsymbol{\psi})&=0,
\quad \forall \boldsymbol{\psi}\in V(g),\\
(\alpha\boldsymbol{\eta}+\boldsymbol{w},\boldsymbol{\psi})&=0,
\quad \forall \boldsymbol{\psi}\in W.
\end{align*}
The third equation yields an algebraic raltion
\begin{align*}
\boldsymbol{\eta}=\frac{-1}{\alpha}\boldsymbol{w}
\end{align*}
between the adjoint $\boldsymbol{w}$ and the control which could be used to eliminate the 
control variable from the problem, by the so called variational discretization approach
\cite{Hinze2005}. However, \cite{Gaspoz2020} suggests that rather than taking this simple substitution a more convenient choice for the following analysis is the consideration of the rescaled adjoint 
\begin{align*}
\boldsymbol{z}=\frac{1}{\sqrt{\alpha}}\boldsymbol{w},\quad
\mbox{~and hence~}\boldsymbol{\eta}=\frac{-1}{\sqrt{\alpha}}\boldsymbol{z}.
\end{align*}
From this it is easy to see that an optimal solution $(\boldsymbol{u},\boldsymbol{\eta})
\in V(g)\times W$ of (\ref{eqn:optimal control problem P}) is equivalently given by a solution $(\boldsymbol{u},\boldsymbol{z})\in V(g)\times V(g)$ of 
\begin{align*}
a(\boldsymbol{u},\boldsymbol{\psi})
+\alpha^{-1/2}(\boldsymbol{z},\boldsymbol{\psi})
&=(\boldsymbol{f},\boldsymbol{\psi}),
\quad \forall \boldsymbol{\psi}\in V(g),\\
a(\boldsymbol{\psi},\boldsymbol{z})
-\alpha^{-1/2}(\boldsymbol{u},\boldsymbol{\psi})
&=-\alpha^{-1/2}(\boldsymbol{u}^*,\boldsymbol{\psi}),
\quad \forall \boldsymbol{\psi}\in V(g).
\end{align*}
Adding the pressure for the divergence constraints, it is also equivalent to seek 
$(\boldsymbol{u},\boldsymbol{z},p,r)\in V\times V\times Q\times Q$
\begin{align}
\label{eqn:PDE problem}
\begin{aligned}
a(\boldsymbol{u},\boldsymbol{\psi})+b(\boldsymbol{\psi},p)
+\alpha^{-1/2}(\boldsymbol{z},\boldsymbol{\psi})
&=(\boldsymbol{f},\boldsymbol{\psi}),
 &&\forall \boldsymbol{\psi}\in V,\\
b(\boldsymbol{u},\phi)&=(g,\phi), &&\forall \phi\in Q,\\
a(\boldsymbol{\psi},\boldsymbol{z})
+b(\boldsymbol{\psi},r)
-\alpha^{-1/2}(\boldsymbol{u},\boldsymbol{\psi})
&=-\alpha^{-1/2}(\boldsymbol{u}^*,\boldsymbol{\psi}),
 &&\forall \boldsymbol{\psi}\in V,\\
b(\boldsymbol{z},\phi)&=(g,\phi), &&\forall \phi\in Q.
\end{aligned} 
\end{align}

\section{Classical FEM discretization}


Let $\mathcal{T}_h$ be a conforming and shape-regular triangulation, which consists of simplices and matches at $\Gamma$ \cite{Layton2002}. The set of edges or faces in $\mathcal{T}_h$ is denoted by $\mathcal{E}_h$. Moreover, we define three useful sub-triangulation as following
\begin{align*}
\mathcal{E}_h(\Gamma)=\{e\in \mathcal{E}_h~|~ e\subset\Gamma\},
\quad \mathcal{T}_h(\Omega^i)=\{T\in \mathcal{T}_h~|~ T\subset\Omega^i\}, \quad i=s,d.
\end{align*}
For $T\in\mathcal{T}_h$, let $h_T$ denote the diameter of the polygon or polyhedra $T$ and $h=\max_{T\in\mathcal{T}_h} h_T$. For the discretization of the fluid's variables in $\Omega^s$, we choose finite element space $V_h^s\subset V^s$, $Q_h^s\subset L_0^2(\Omega^s)$ introduced in Section 4.1.1 in \cite{Linke2016}, which are LBB-stable
and satisfy a discrete Korn inequality.
This velocity space is achieved by enriching the space $P_k,k\geq 2$, of continuous, piece-wise polynomial functions of degree less than or equal to $k$ with suitable bubble functions. 
And the pressure space contains the functions with degree $k-1$ and vanishing integral in $\Omega^s$.

For the porous medium problem in $\Omega^d$, 
we choose mixed finite element spaces to be the $RT_{k-1}$ spaces \cite{Raviart1977} for the normal velocity vanishing on boundary $\Gamma^d$, and the pressure space similar to fluid domain, denoted by $V_h^d$  and $Q_h^d$. It is known for these choices that $V_h^d\subset V^d, Q_h^d\in L_0^2(\Omega^d)$, and
$\nabla\cdot V_h^d=Q_h^d$.

On the interface $\Gamma$, define another finite element space, which is the normal trace of $V_h^d$ on $\Gamma_h$, as following
\begin{align*}
\Lambda_h=\{\lambda_h\in L^2(\Gamma)~|~\mu_{h|e} \in P_{k-1}(e), \forall e\in\mathcal{E}_h(\Gamma)\}.
\end{align*}

With these spaces $V_h^s, V_h^d$, and $\Lambda_h$, define
\begin{align}\label{eqn:definition of Vh}
V_h=\{\boldsymbol{v}_h=(\boldsymbol{v}_h^s,\boldsymbol{v}_h^d)\in V_h^s\times V_h^d~|~
\langle\boldsymbol{v}_h^s\cdot\boldsymbol{n}^s
+\boldsymbol{v}_h^d\cdot\boldsymbol{n}^d,\lambda_h\rangle_{\Gamma}=0,\forall \lambda_h\in\Lambda_h\}.
\end{align}
Note that, sice function $\lambda_h\in\Lambda_h$ does not in general vanish on $\partial\Gamma$ ($\Lambda_h\not\subset\Lambda$), the space $V_h$ is nonconforming ($V_h\not\subset V$).
With spaces $Q_h^s$ and $Q_h^d$, define $Q_h=Q_h^s\times Q_h^d$, which is vanishing integral in $\Omega$ and satisfy $Q_h\subset Q$.
For the discretization with respect to the space $V(g)$, define
\begin{align*}
V_h(g)=\{\boldsymbol{\psi}\in V_h~|~b(\boldsymbol{\psi}_h,\phi_h)=(g,\phi_h),~\forall \phi_h\in Q_h\}
\end{align*}
for the inf-sup stable pair $V_h\times Q_h$ with $V_h(g)\not\subset V(g)$ 
(since the divergence-free and normal continuity across $\Gamma$ hold discretely).

The analysis  is based on a functional on $V_h(0)$, which consists of relaxed divergence-free constraint and relaxed normal continuity across $\Gamma$. This functional is defined by
\begin{align}
\label{eqn:consistency error functional}
\vartheta_{\phi}(\boldsymbol{\psi}_h)=b(\boldsymbol{\psi}_h,\phi)
-\langle\boldsymbol{\psi}_h^s\cdot\boldsymbol{n}^s
+\boldsymbol{\psi}_h^d\cdot\boldsymbol{n}^d,\phi^d\rangle_{\Gamma}
, \quad \forall \boldsymbol{\psi}_h\in V_h(0),
\end{align}
related to $\phi=(\phi^s,\phi^d)\in Q$ and $\phi^d\in H^1(\Omega^d)$. Then the consistency error (the measure of $V_h$ being not a subset of $V$)
in form of the dual form is defined by
\begin{align}\label{eqn:consistency error subset}
\|\vartheta_{\phi}\|_{(V_h(0))^*}=\sup_{\boldsymbol{\psi}_h\in V_h(0)}
\frac{\vartheta_{\phi}(\boldsymbol{\psi}_h)}
{\|\boldsymbol{\psi}_h\|_X}.
\end{align}
For the function exactly satisfying divergence constraint and interface normal continuity,  i.e. $\boldsymbol{\psi}_h\in V_h(0)\cap V(0)$, this dual norm is always zero. 
\begin{remark}
In this work, we use $\mathbbm{a}\lesssim \mathbbm{b}$ when there exists a constant $c$ independent of $\mathbbm{a},\mathbbm{b},h$ such that $\mathbbm{a}\leq c\mathbbm{b}$.
\end{remark}

\begin{lemma}
\label{lem:consistency error}
For any $\phi=(\phi^s,\phi^d)\in Q$ with $\phi^d\in H^1(\Omega^d)$, the induced consistency error defined in (\ref{eqn:consistency error functional}),
is well-defined, and with the estimation
\begin{align*}
\|\vartheta_{\phi}\|_{(V_h(0))^*}^2
\lesssim &
\inf_{\phi_h\in Q_h^s}\|\phi^s-\phi_h\|_s^2
+h\inf_{\phi_h\in\Lambda_h}\|\phi^d-\phi_h\|_{\Gamma}^2.
\end{align*}
\end{lemma}
\begin{proof}
We give the boundedness of the linear functionals on the space $V_h(0)$ in (\ref{eqn:consistency error subset}). To begin with, let $\lambda_h\in\Lambda_h$ to be the local $L^2$ projection of $\phi_{|\Gamma}^d$ into $\Lambda_h$, i.e. $\langle \phi^d-\lambda_h,\phi_h\rangle_e=0$ for any $\phi_h\in P_{k-1}(e)$, $e\in\mathcal{E}_h(\Gamma)$. For any function $\boldsymbol{\psi}_h\in V_h(0)$, noting the fact $\nabla\cdot\boldsymbol{\psi}_h^d=0$ and $\boldsymbol{\psi}_h^d\cdot\boldsymbol{n}^d_{|\Gamma}\in\Lambda_h$, it holds 
\begin{align}\label{eqn:psi_d estimate}
\langle\boldsymbol{\psi}_h^d\cdot\boldsymbol{n}^d,\phi^d-\lambda_h\rangle_{\Gamma}=0.
\end{align}
Let $\bar{\boldsymbol{\psi}}_h^s$ to be the local $L^2$ projection of $\boldsymbol{\psi}_h^s$ into $\{\boldsymbol{v}_h\in [L^2(\Omega^s)]^N~|~\boldsymbol{v}_{h|T}\in [P_{k-1}(T)]^N,~T\in \mathcal{T}_h(\Omega^s)\}$. From $\bar{\boldsymbol{\psi}}_h^s\cdot\boldsymbol{n}_{|\Gamma}^s\in \Lambda_h$, Cauchy-Schwarz inequality, trace inequality, and projection property, we have
\begin{align}\label{eqn:psi_s estimate}
\begin{aligned}
&\langle\boldsymbol{\psi}_h^s\cdot\boldsymbol{n}^s
,\phi^d-\lambda_h\rangle_{\Gamma}\\
=&\langle\boldsymbol{\psi}_h^s\cdot\boldsymbol{n}^s
-\bar{\boldsymbol{\psi}}_h^s\cdot\boldsymbol{n}^s
,\phi^d-\lambda_h\rangle_{\Gamma}\\
\leq&\|\boldsymbol{\psi}_h^s\cdot\boldsymbol{n}^s
-\bar{\boldsymbol{\psi}}_h^s\cdot\boldsymbol{n}^s\|_{\Gamma}
\inf_{\phi_h\in\Lambda_h}\|\phi^d-\phi_h\|_{\Gamma}\\
\leq&\left(\sum_{T\in\mathcal{T}_h(\Gamma^s)}h_T^{-1}
\|\boldsymbol{\psi}_h^s-\bar{\boldsymbol{\psi}}_h^s\|_{0,T}^2
+h_T|\boldsymbol{\psi}_h^s-\bar{\boldsymbol{\psi}}_h^s|_{1,T}^2\right)^{1/2}
\inf_{\phi_h\in\Lambda_h}\|\phi^d-\phi_h\|_{\Gamma}\\
\leq&\left(\sum_{T\in\mathcal{T}_h(\Gamma^s)}h_T|\boldsymbol{\psi}_h^s|_{1,T}^2\right)^{1/2}
\inf_{\phi_h\in\Lambda_h}\|\phi^d-\phi_h\|_{\Gamma},
\end{aligned}
\end{align}
where $\mathcal{T}_h(\Gamma^s)=\{T\in\mathcal{T}_h(\Omega^s)~|~T\cap\Gamma\neq\emptyset\}$. 

Then, combining with (\ref{eqn:definition of Vh}), (\ref{eqn:psi_d estimate}), and (\ref{eqn:psi_s estimate}),
it can be obtained
\begin{align*}
\vartheta_{\phi}(\boldsymbol{\psi}_h)
=&-(\nabla\cdot\boldsymbol{\psi}_h^s,\phi^s)_s
-(\nabla\cdot\boldsymbol{\psi}_h^d,\phi^d)_d
-\langle\boldsymbol{\psi}_h^s\cdot\boldsymbol{n}^s
+\boldsymbol{\psi}_h^d\cdot\boldsymbol{n}^d,\phi^d\rangle_{\Gamma}\\
=&-(\nabla\cdot\boldsymbol{\psi}_h^s,\phi^s)_s
-\langle\boldsymbol{\psi}_h^s\cdot\boldsymbol{n}^s
,\phi^d-\lambda_h\rangle_{\Gamma}\\
\leq &\|\nabla\boldsymbol{\psi}_h^s\|_s\inf_{\phi_h\in Q_h^s}\|\phi^s-\phi_h\|_s
+\left(\sum_{T\in\mathcal{T}_h(\Gamma^s)}h_T|\boldsymbol{\psi}_h^s|_{1,T}^2\right)^{1/2}
\inf_{\phi_h\in\Lambda_h}\|\phi^d-\phi_h\|_{\Gamma}\\
\leq& \left(\inf_{\phi_h\in Q_h^s}\|\phi^s-\phi_h\|_s^2+
h\inf_{\phi_h\in\Lambda_h}\|\phi^d-\phi_h\|_{\Gamma}^2\right)^{1/2}\|\boldsymbol{\psi}_h\|_X,
\end{align*}
which complete the proof.
\end{proof}


It is studied how the consistency errors from the lack of pressure-robustness 
for the natural coupled energy norm induced by the PDE rather than the cost functional, see also \cite{Gaspoz2020},
\begin{align}\label{eqn:X norm for u and z}
\|(\boldsymbol{v},\boldsymbol{\psi})\|_X^2
=\|\boldsymbol{v}\|_X^2+\|\boldsymbol{\psi}\|_X^2, 
\mbox{~for~} \boldsymbol{v},\boldsymbol{\psi}\in V^s\times V^d.
\end{align}
To do so, we estimate the distance of $(\boldsymbol{u}_h,\boldsymbol{z}_h)$ to the Stokes-Darcy best-approximation, denoted by $(S_h\boldsymbol{u},S_h\boldsymbol{z})\in V_h\times V_h$, that is defined by
\begin{align}\label{eqn:Sh definition}
\begin{aligned}
a(S_h\boldsymbol{v},\boldsymbol{\psi}_h)+b(\boldsymbol{\psi}_h,\delta_h)&= a(\boldsymbol{v},\boldsymbol{\psi}_h),\quad \forall\boldsymbol{\psi}_h\in V_h,\\
b(S_h\boldsymbol{v},\phi_h)&=(g,\phi_h),\quad \forall \phi_h\in Q_h.
\end{aligned}
\end{align}
for $\boldsymbol{v}=\boldsymbol{u}$ or $\boldsymbol{z}$, respectively. Here $\delta_h$ denotes the pressure, and we have referred to a similar approach to the Stokes optimal control problem \cite{merdonPressureRobustnessContextOptimal2023}.
Then for any $\boldsymbol{v}_h\in V_h$, from (\ref{eqn:Sh definition}) and noting $(g,\phi_h)=b(\boldsymbol{v},\phi_h)$, we have 
\begin{align*}
a(S_h\boldsymbol{v}-\boldsymbol{v}_h,\boldsymbol{\psi}_h)+b(\boldsymbol{\psi}_h, \delta_h)&=a(\boldsymbol{v}-\boldsymbol{v}_h,\boldsymbol{\psi}_h),\quad \forall\boldsymbol{\psi}_h\in V_h,\\
b(S_h\boldsymbol{v}-\boldsymbol{v}_h,\phi_h)&=b(\boldsymbol{v}-\boldsymbol{v}_h,\phi_h),\quad \forall \phi_h\in Q_h.
\end{align*}
Since $(V_h,Q_h)$ is LBB-stable shown in Lemma~4.3 in \cite{Layton2002}, and operator $a(\cdot,\cdot)$ is coercive for $V^s\times\{\boldsymbol{\psi}^d\in V^d~|~\nabla\cdot\boldsymbol{\psi}^d=0\}$ shown in Lemma~3.1 in \cite{Layton2002}, then the abstract bound estimates (Theorem~1.2 in \cite{Brezzi1991}) implies, for $\boldsymbol{v}=\boldsymbol{u}$ or $\boldsymbol{z}$, 
\begin{align*}
\|S_h\boldsymbol{v}-\boldsymbol{v}_h\|_X\lesssim \|\boldsymbol{v}-\boldsymbol{v}_h\|_X,
\end{align*}
which equals 
\begin{align}\label{eqn:error u and Sh u}
\|\boldsymbol{v}-S_h\boldsymbol{v}\|_X\leq \inf_{\boldsymbol{v}_h\in V_h}(\|\boldsymbol{v}-\boldsymbol{v}_h\|_X+\|S_h\boldsymbol{v}-\boldsymbol{v}_h\|_X)
\lesssim \inf_{\boldsymbol{v}_h\in V_h}\|\boldsymbol{v}-\boldsymbol{v}_h\|_X.
\end{align} 
This shows convergence rates corresponding to the regularity of $\boldsymbol{u}$ and $\boldsymbol{z}$ and the polynomial order of $V_h$. From the triangle inequality, this approximation result 
$\|(\boldsymbol{u}-\boldsymbol{u}_h,\boldsymbol{z}-\boldsymbol{z}_h)\|_X^2$ is only perturbed by the term $\|(S_h\boldsymbol{u}-\boldsymbol{u}_h,S_h\boldsymbol{z}-\boldsymbol{z}_h)\|_X^2$ which therefore is the primal object of interest in the a priori error analysis below.


The classical variational discretization of (\ref{eqn:PDE problem}) solves the following discrete problem: seek $(\boldsymbol{u}_h,\boldsymbol{z}_h,p_h,r_h)\in V_h\times V_h\times Q_h\times Q_h$ such that
\begin{align}
\label{eqn:PDE problem clssical}
\begin{aligned}
a(\boldsymbol{u}_h,\boldsymbol{\psi}_h)+b(\boldsymbol{\psi}_h,p_h)
+\alpha^{-1/2}(\boldsymbol{z}_h,\boldsymbol{\psi}_h)
&=(\boldsymbol{f},\boldsymbol{\psi}_h),
 &&\forall \boldsymbol{\psi}_h\in V_h,\\
b(\boldsymbol{u}_h,\phi_h)&=(g,\phi_h), &&\forall \phi_h\in Q_h,\\
a(\boldsymbol{\psi}_h,\boldsymbol{z}_h)
+b(\boldsymbol{\psi}_h,r_h)
-\alpha^{-1/2}(\boldsymbol{u}_h,\boldsymbol{\psi}_h)
&=-\alpha^{-1/2}(\boldsymbol{u}^*,\boldsymbol{\psi}_h),
 &&\forall \boldsymbol{\psi}_h\in V_h,\\
b(\boldsymbol{z}_h,\phi_h)&=(g,\phi_h), &&\forall \phi_h\in Q_h.
\end{aligned} 
\end{align}
The error estimates involving the previously defined projector $S_h:V(g)\rightarrow V_h(g)$ in (\ref{eqn:Sh definition}) needs to be discussed in the following lemma.
\begin{theorem}\label{the:classical estimate}
If the exact solutions of (\ref{eqn:PDE problem}) have the regularity
$\boldsymbol{u}^s,\boldsymbol{z}^s\in [H^2(\Omega^s)]^N$, and $p^i,r^i\in H^1(\Omega^i), i=s,d$,
it holds
\begin{align}\label{eqn:error estimate pressure depend}
\begin{aligned}
\|(S_h\boldsymbol{u}-\boldsymbol{u}_h,S_h\boldsymbol{z}-\boldsymbol{z}_h)\|_X^2
\lesssim &\|\vartheta_p\|_{(V_h(0))^*}^2
+\|\vartheta_r\|_{(V_h(0))^*}^2\\
&+\|\boldsymbol{u}-S_h\boldsymbol{u}\|_X^2
+\|\boldsymbol{z}-S_h\boldsymbol{z}\|_X^2.
\end{aligned}
\end{align}
for the solution of $(\boldsymbol{u}_h,\boldsymbol{z}_h)$ of (\ref{eqn:PDE problem clssical}) and 
the discrete approximation $(S_h\boldsymbol{u},S_h\boldsymbol{z})$ defined in (\ref{eqn:Sh definition}) of 
the exact solutions $(\boldsymbol{u},\boldsymbol{z})$ of (\ref{eqn:PDE problem}).
More generally, if the exact solutions of have the regularity
$(\boldsymbol{u}^s,\boldsymbol{u}^d),(\boldsymbol{z}^s,\boldsymbol{z}^d)\in [H^k(\Omega^s)]^N\times [H^{k-1}(\Omega^d)]^N$, $\nabla\cdot \boldsymbol{u}^d, \nabla\cdot \boldsymbol{z}^d\in H^{k-1}(\Omega^d)$, and $p^i,r^i\in H^{k-1}(\Omega^i), i=s,d$ for any $k\geq 2$,
it holds
\begin{align}\label{eqn:error estimate regular pressure depend}
\begin{aligned}
\|(\boldsymbol{u}-\boldsymbol{u}_h,\boldsymbol{z}-\boldsymbol{z}_h)\|_X^2
\lesssim h^{2(k-1)}(&\|\boldsymbol{u}^s\|_{k,s}^2+\|\boldsymbol{u}^d\|_{k-1,d}^2+\|\nabla\cdot\boldsymbol{u}^d\|_{k-1,d}^2\\
+&\|\boldsymbol{z}^s\|_{k,s}^2+\|\boldsymbol{z}^d\|_{k-1,d}^2+\|\nabla\cdot\boldsymbol{z}^d\|_{k-1,d}^2\\
+&\|p^s\|_{k-1,s}^2+\|p^d\|_{k-1,d}^2+\|r^s\|_{k-1,s}^2+\|r^d\|_{k-1,d}^2).
\end{aligned}
\end{align}
\end{theorem}
\begin{proof}

For any $\boldsymbol{\psi}_h\in V_h(g)$, according to the definition of $a(\cdot,\cdot)$ in (\ref{eqn:state weak 1}), we have
\begin{align}
\label{eqn:u and psi_h}
a(\boldsymbol{u},\boldsymbol{\psi}_h)
&=2\mu(D(\boldsymbol{u}^s),D(\boldsymbol{\psi}_h^s))_s
+\mu(K^{-1}\boldsymbol{u}^d,\boldsymbol{\psi}_h^d)
+\sum_{j=1}^{N-1}\frac{\alpha_1\mu}{\sqrt{\kappa_j}}\langle\boldsymbol{u}^s\cdot\boldsymbol{\tau}_j
,\boldsymbol{\psi}_h^s\cdot\boldsymbol{\tau}_j\rangle_{\Gamma}.
\end{align}
Since $\boldsymbol{u}^s\in [H^2(\Omega^s)]^N$ and $p^i\in H^1(\Omega^i), i=s,d$, the solutions $(\boldsymbol{u},p)$ of (\ref{eqn:PDE problem}) satisfy (\ref{eqn:state Stokes})$\sim$(\ref{eqn:state boundary 3}) by replacing $\boldsymbol{f}$ with 
$\boldsymbol{f}-\alpha^{-1/2}\boldsymbol{\eta}$.
Then, based on the usual integration by parts and the boundary condition (\ref{eqn:state boundary 3}), the first and the third summands in (\ref{eqn:u and psi_h}) can be derived into
\begin{align}
\label{eqn:the first summand}
&2\mu(D(\boldsymbol{u}^s),D(\boldsymbol{\psi}_h^s))_s=-2\mu(\nabla\cdot D(\boldsymbol{u}^s),\boldsymbol{\psi}_h^s)_s
+2\mu\langle D(\boldsymbol{u}^s)\boldsymbol{n}^s,\boldsymbol{\psi}_h^s\rangle_{\Gamma},\\
\label{eqn:the third summand}
&\sum_{j=1}^{N-1}\frac{\alpha_1\mu}{\sqrt{\kappa_j}}\langle\boldsymbol{u}^s\cdot\boldsymbol{\tau}_j
,\boldsymbol{\psi}_h^s\cdot\boldsymbol{\tau}_j\rangle_{\Gamma}
=\sum_{j=1}^{N-1}-2\mu\langle D(\boldsymbol{u}^s)\boldsymbol{n}^s\cdot\boldsymbol{\tau}_j
,\boldsymbol{\psi}_h^s\cdot\boldsymbol{\tau}_j\rangle_{\Gamma}.
\end{align}

From (\ref{eqn:u and psi_h}), (\ref{eqn:the first summand}), (\ref{eqn:the third summand}), and noting the Stokes-Darcy equations in (\ref{eqn:state Stokes}) and (\ref{eqn:state Darcy}), and basing on the equality
\begin{align*} 
2\mu\langle D(\boldsymbol{u}^s)\boldsymbol{n}^s,\boldsymbol{\psi}_h^s\rangle_{\Gamma}
=\sum_{j=1}^{N-1}2\mu\langle D(\boldsymbol{u}^s)\boldsymbol{n}^s\cdot\boldsymbol{\tau}_j
,\boldsymbol{\psi}_h^s\cdot\boldsymbol{\tau}_j\rangle_{\Gamma}
+2\mu\langle D(\boldsymbol{u}^s)\boldsymbol{n}^s\cdot\boldsymbol{n}^s
,\boldsymbol{\psi}_h^s\cdot\boldsymbol{n}^s\rangle_{\Gamma},
\end{align*}
we have
\begin{align*}
&a(\boldsymbol{u},\boldsymbol{\psi}_h)\\
=&(\boldsymbol{f},\boldsymbol{\psi}_h)
-\alpha^{-1/2}(\boldsymbol{z},\boldsymbol{\psi}_h)
-(\nabla p^s,\boldsymbol{\psi}_h^s)_s
-(\nabla p^d,\boldsymbol{\psi}_h^d)_d
+2\mu\langle D(\boldsymbol{u}^s)\boldsymbol{n}^s\cdot\boldsymbol{n}^s
,\boldsymbol{\psi}_h^s\cdot\boldsymbol{n}^s\rangle_{\Gamma}.
\end{align*}
By using the integration by parts for the third and fourth term in the above equation, and nothing the boundary condition in (\ref{eqn:state boundary 2}), it can be obtained that
\begin{align*}
&-(\nabla p^s,\boldsymbol{\psi}_h^s)_s
-(\nabla p^d,\boldsymbol{\psi}_h^d)_d\\
=&(p^s,\nabla\cdot\boldsymbol{\psi}_h^s)_s
-\langle p^s,\boldsymbol{\psi}_h^s\cdot\boldsymbol{n}^s\rangle_{\Gamma}
+(p^d,\nabla\cdot\boldsymbol{\psi}_h^d)_d
-\langle p^d,\boldsymbol{\psi}_h^d\cdot\boldsymbol{n}^d\rangle_{\Gamma}\\
=&-b(\boldsymbol{\psi}_h,p)-2\mu\langle D(\boldsymbol{u}^s)\boldsymbol{n}^s\cdot\boldsymbol{n}^s
,\boldsymbol{\psi}_h^s\cdot\boldsymbol{n}^s\rangle_{\Gamma}
-\langle p^d,\boldsymbol{\psi}_h^s\cdot\boldsymbol{n}^s
+\boldsymbol{\psi}_h^d\cdot\boldsymbol{n}^d\rangle_{\Gamma}
\end{align*}
Thus, combining with (\ref{eqn:consistency error functional}), we have
\begin{align}
\label{eqn:u and phi_h simple form}
\begin{aligned}
a(\boldsymbol{u},\boldsymbol{\psi}_h)
=&(\boldsymbol{f},\boldsymbol{\psi}_h)
-\alpha^{-1/2}(\boldsymbol{z},\boldsymbol{\psi}_h)
-b(\boldsymbol{\psi}_h,p)
-\langle p^d,\boldsymbol{\psi}_h^s\cdot\boldsymbol{n}^s
+\boldsymbol{\psi}_h^d\cdot\boldsymbol{n}^d\rangle_{\Gamma}\\
=&(\boldsymbol{f},\boldsymbol{\psi}_h)
-\alpha^{-1/2}(\boldsymbol{z},\boldsymbol{\psi}_h)
-\vartheta_p(\boldsymbol{\psi}_h).
\end{aligned}
\end{align}

From (\ref{eqn:Sh definition}), it is clear that $S_h\boldsymbol{u}\in V_h(g)$. Since $S_h\boldsymbol{u}-\boldsymbol{u}_h\in V_h(0)$ and $\nabla\cdot V_h^d=Q_h^d$, it can be obtained $\nabla\cdot (S_h\boldsymbol{u}-\boldsymbol{u}_h)^d=0$. Then, from Lemma~3.1 in \cite{Layton2002}, (\ref{eqn:PDE problem clssical}), and (\ref{eqn:u and phi_h simple form}), it reveals
\begin{align*}
\|S_h\boldsymbol{u}-\boldsymbol{u}_h\|_X^2\lesssim & a(S_h\boldsymbol{u}-\boldsymbol{u}_h,S_h\boldsymbol{u}-\boldsymbol{u}_h)\\
= &a(S_h\boldsymbol{u}-\boldsymbol{u},S_h\boldsymbol{u}-\boldsymbol{u}_h)
+a(\boldsymbol{u}-\boldsymbol{u}_h,S_h\boldsymbol{u}-\boldsymbol{u}_h)\\
= &a(S_h\boldsymbol{u}-\boldsymbol{u},S_h\boldsymbol{u}-\boldsymbol{u}_h)
+(\boldsymbol{f}-\alpha^{-1/2}\boldsymbol{z},S_h\boldsymbol{u}-\boldsymbol{u}_h)
-\vartheta_p(S_h\boldsymbol{u}-\boldsymbol{u}_h)\\
&-(\boldsymbol{f}-\alpha^{-1/2}\boldsymbol{z}_h,S_h\boldsymbol{u}-\boldsymbol{u}_h)\\
= &a(S_h\boldsymbol{u}-\boldsymbol{u},S_h\boldsymbol{u}-\boldsymbol{u}_h)
-\alpha^{-1/2}(\boldsymbol{z}-\boldsymbol{z}_h,S_h\boldsymbol{u}-\boldsymbol{u}_h)
-\vartheta_p(S_h\boldsymbol{u}-\boldsymbol{u}_h)
\end{align*}
Analogously, one shows that
\begin{align*}
\|S_h\boldsymbol{z}-\boldsymbol{z}_h\|_X^2\lesssim 
a(S_h\boldsymbol{z}-\boldsymbol{z},S_h\boldsymbol{z}-\boldsymbol{z}_h)
+\alpha^{-1/2}(\boldsymbol{u}-\boldsymbol{u}_h,S_h\boldsymbol{z}-\boldsymbol{z}_h)
-\vartheta_r(S_h\boldsymbol{z}-\boldsymbol{z}_h)
\end{align*}

From the boundness of $a(\cdot,\cdot)$, we have
\begin{align*}
I_1=&a(S_h\boldsymbol{u}-\boldsymbol{u},S_h\boldsymbol{u}-\boldsymbol{u}_h)
+a(S_h\boldsymbol{z}-\boldsymbol{z},S_h\boldsymbol{z}-\boldsymbol{z}_h)\\
\lesssim & \|(S_h\boldsymbol{u}-\boldsymbol{u},S_h\boldsymbol{z}-\boldsymbol{z})\|_X
\|(S_h\boldsymbol{u}-\boldsymbol{u}_h,,S_h\boldsymbol{z}-\boldsymbol{z}_h)\|_X.
\end{align*}
Then, using
\begin{align*}
(\boldsymbol{z}-\boldsymbol{z}_h,S_h\boldsymbol{u}-\boldsymbol{u}_h)
&=(S_h\boldsymbol{z}-\boldsymbol{z}_h,S_h\boldsymbol{u}-\boldsymbol{u}_h)
+(\boldsymbol{z}-S_h\boldsymbol{z},S_h\boldsymbol{u}-\boldsymbol{u}_h)\\
(\boldsymbol{u}-\boldsymbol{u}_h,S_h\boldsymbol{z}-\boldsymbol{z}_h)
&=(S_h\boldsymbol{u}-\boldsymbol{u}_h,S_h\boldsymbol{z}-\boldsymbol{z}_h)
+(\boldsymbol{u}-S_h\boldsymbol{u},S_h\boldsymbol{z}-\boldsymbol{z}_h)
\end{align*}
one obtains
\begin{align*}
I_2:&=(\boldsymbol{z}-\boldsymbol{z}_h,S_h\boldsymbol{u}-\boldsymbol{u}_h)
-(\boldsymbol{u}-\boldsymbol{u}_h,S_h\boldsymbol{z}-\boldsymbol{z}_h)\\
&=(\boldsymbol{z}-S_h\boldsymbol{z},S_h\boldsymbol{u}-\boldsymbol{u}_h)
-(\boldsymbol{u}-S_h\boldsymbol{u},S_h\boldsymbol{z}-\boldsymbol{z}_h)
\end{align*}
which can be estimated by, from the Poincar\'e-Friedrichs inequality (Lemma B.66 in \cite{Ern2004}), 
\begin{align*}
|I_2|&\leq (\|\boldsymbol{u}-S_h\boldsymbol{u}\|^2+\|\boldsymbol{z}-S_h\boldsymbol{z}\|^2)^{1/2}
(\|S_h\boldsymbol{u}-\boldsymbol{u}_h\|^2+\|S_h\boldsymbol{z}-\boldsymbol{z}_h\|^2)^{1/2}\\
&\lesssim (\|\boldsymbol{u}-S_h\boldsymbol{u}\|^2+\|\boldsymbol{z}-S_h\boldsymbol{z}\|^2)^{1/2}
(\|S_h\boldsymbol{u}-\boldsymbol{u}_h\|_X^2+\|S_h\boldsymbol{z}-\boldsymbol{z}_h\|_X^2)^{1/2}.
\end{align*}
Analogously, for
\begin{align*}
I_3:&=\vartheta_p(S_h\boldsymbol{u}-\boldsymbol{u}_h)
+\vartheta_r(S_h\boldsymbol{z}-\boldsymbol{z}_h)
\end{align*}
one obtains the estimate
\begin{align*}
|I_3|\leq \left(\|\vartheta_p\|_{(V_h(0))^*}^2+\|\vartheta_r\|_{(V_h(0))^*}^2\right)^{1/2}
\left(\|S_h\boldsymbol{u}-\boldsymbol{u}_h\|_X^2
+\|S_h\boldsymbol{z}-\boldsymbol{z}_h\|_X^2\right)^{1/2}
\end{align*}
Combining with $\|(S_h\boldsymbol{u}-\boldsymbol{u}_h,S_h\boldsymbol{z}-\boldsymbol{z}_h)\|_X^2
=I_1+I_2+I_3$, we get (\ref{eqn:error estimate pressure depend}).

Next, we will prove (\ref{eqn:error estimate regular pressure depend}). If $p^i\in H^{k-1}(\Omega^i), i=s,d$ for any $k\geq 2$, let $\phi_h^d$ to be the local $L^2$ projection of $p^d$ into $\{q_h\in L^2(\Omega^d)~|~q_{h|T}\in P_{k-1}(T),~T\in \mathcal{T}_h(\Omega^d)\}$.
From Lemma~\ref{lem:consistency error}, projection property, and trace inequality, we have
\begin{align}\label{eqn:cosistency error p}
\begin{aligned}
\|\vartheta_p\|_{(V_h(0))^*}^2
\lesssim &
\inf_{\phi_h\in Q_h^s}\|p^s-\phi_h\|_s^2
+h\inf_{\phi_h\in\Lambda_h}\|p^d-\phi_h\|_{\Gamma}^2\\
\lesssim & h^{2(k-1)}\|p^s\|_{k-1,s}^2+ h\|p^d-\phi_h^d\|_{\Gamma}^2\\
\lesssim & h^{2(k-1)}\|p^s\|_{k-1,s}^2+ h\sum_{T\in\mathcal{T}_h(\Gamma)}h_T^{-1}\|p^d-\phi_h^d\|_T^2+h_T\|\nabla(p^d-\phi_h^d)\|_T^2\\
\lesssim & h^{2(k-1)}\|p^s\|_{k-1,s}^2+h\sum_{T\in\mathcal{T}_h(\Gamma)}h_T^{2(k-1)-1}\|p^d\|_{k-1,T}^2\\
\lesssim & h^{2(k-1)}(\|p^s\|_{k-1,s}^2+\|p^d\|_{k-1,d}^2),
\end{aligned}
\end{align}
where $\mathcal{T}_h(\Gamma)$ is defined in ($\ref{eqn:psi_s estimate}$). Similarly, $\|\vartheta_r\|_{(V_h(0))^*}$ has the following estimate
\begin{align}\label{eqn:cosistency error r}
\|\vartheta_r\|_{(V_h(0))^*}^2
\lesssim & h^{2(k-1)}(\|r^s\|_{k-1,s}^2+\|r^d\|_{k-1,d}^2).
\end{align}

For the terms $\|\boldsymbol{v}-S_h\boldsymbol{v}\|_X^2$ in (\ref{eqn:error estimate pressure depend}), with $\boldsymbol{v}=\boldsymbol{u}$ or $\boldsymbol{z}$, if $(\boldsymbol{v}^s,\boldsymbol{v}^d)\in [H^k(\Omega^s)]^N\times [H^{k-1}(\Omega^d)]^N$ and $\nabla\cdot \boldsymbol{v}^d\in H^{k-1}(\Omega^d)$ for any $k\geq 2$, from the approximation property (\ref{eqn:error u and Sh u}) and the interpolation properties introduced in (4.3), (4.13), (4.14), and (4.35) in \cite{Layton2002}, it is obtained
\begin{align}\label{eqn:approximation error u and z}
\begin{aligned}
\|\boldsymbol{v}-S_h\boldsymbol{v}\|_X^2\lesssim &\inf_{\boldsymbol{v}_h\in V_h}\|\boldsymbol{v}-\boldsymbol{v}_h\|_X^2\\
=& \inf_{\boldsymbol{v}_h\in V_h}|\boldsymbol{v}^s-\boldsymbol{v}_h^s|_{1,s}^2+\|\boldsymbol{v}^d-\boldsymbol{v}_h^d\|_d^2
+\|\nabla\cdot(\boldsymbol{v}^d-\boldsymbol{v}_h^d)\|_d^2\\
\lesssim & h^{2(k-1)}(\|\boldsymbol{v}^s\|_{k,s}^2+\|\boldsymbol{v}^d\|_{k-1,d}^2+\|\nabla\cdot \boldsymbol{v}^d\|_{k-1,d}^2).
\end{aligned}
\end{align}
Combining with (\ref{eqn:cosistency error p}), (\ref{eqn:cosistency error r}), and (\ref{eqn:approximation error u and z}), it yields (\ref{eqn:error estimate regular pressure depend}).
\end{proof}
\begin{remark}
The pressure dependency of dual norms $\|\vartheta_p\|_{(V_h(0))^*}^2$ and $\|\vartheta_r\|_{(V_h(0))^*}^2$ come from the presence of a function in $V_h(0)$, which either does not satisfy exactly divergence-free or does not satisfy interface normal continuity.
\end{remark}
\section{Pressure-robust discretization}
In this section,  we will apply a reconstruction operator to the discretization scheme to achieve exactly divergence-free and interface normal continuity, thereby obtaining a pressure-robust discretization. Therefore, for any $T\in\mathcal{T}_h$, we consider the Raviart-Thomas space $RT_{k-1}(T)=[P_{k-1}(T)]^N+\boldsymbol{x}P_{k-1}(T)$ for $k\geq 2$. We get for $\boldsymbol{v}_h\in RT_{k-1}(T)$ that $\nabla\cdot\boldsymbol{v}_h\in P_{k-1}(T)$,
$\boldsymbol{v}_h\cdot\boldsymbol{n}|_e\in P_{k-1}(e), e\subset\partial T$.
Then, define
\begin{align*}
\Upsilon_h^i=\{\boldsymbol{v}_h\in H(div;\Omega^i)~|~\boldsymbol{v}_{h|T}\in RT_{k-1}(T),\forall T\in\mathcal{T}_h, \boldsymbol{v}_h\cdot \boldsymbol{n}=0 \mbox{~on~}\Gamma^i\},
\end{align*}
with $i=s,d$, respectively. 
Note that $\Upsilon_h^d=V_h^d$.
Then, we define reconstruction operators $\Pi^i_h: V^i\rightarrow \Upsilon_h^i$ locally by 
\begin{align}
\label{eqn:Pi 1}
&(\Pi_h^i\boldsymbol{v}-\boldsymbol{v},\boldsymbol{\psi}_h)_T=0,\quad \boldsymbol{\psi}_h\in [P_{k-2}(T)]^N,\\
\label{eqn:Pi 2}
&\langle(\Pi_h^i\boldsymbol{v}-\boldsymbol{v})\cdot\boldsymbol{n},q_h\rangle_e,\quad
q_h\in P_{k-1}(e),
e\subset \partial T,
\end{align}
for any element  $T\subset\Omega^i$. From (2.5.10) and Proposition 2.5.1 in \cite{Boffi2013}, the reconstruction operators are well defined with the property
\begin{align}\label{eqn:Pi inequality}
\|\boldsymbol{v}-\Pi_h^i\boldsymbol{v}\|_{\ell,T}\lesssim h_T^{m-\ell}|\boldsymbol{v}|_{m,T},\quad  \ell=0,1,
\end{align}
for any $\boldsymbol{v}\in H^m(\Omega^i)$ and $1\leq m\leq k$.

\begin{lemma}\label{lem:Pi}
Define $\Pi_h=\Pi_h^s\times \Pi_h^d: V^s\times V^d\rightarrow \Upsilon_h=\Upsilon_h^s\times \Upsilon_h^d$, which
has the following properties
\begin{align}
&\Pi_h:V_h(0)\rightarrow \Upsilon_h\cap \Upsilon.
\end{align}
where
\begin{align*}
\Upsilon=\{\boldsymbol{\psi}\in H(div,\Omega)~|~ \nabla\cdot\boldsymbol{\psi}=0,~ 
&(\boldsymbol{\psi} _s\cdot\boldsymbol{n}^s)_{|\Gamma^s}=0, ~
(\boldsymbol{\psi} _d\cdot\boldsymbol{n}^d)_{|\Gamma^d}=0, \\ \mbox{~and~}&(\boldsymbol{\psi}^s\cdot\boldsymbol{n}^s
+\boldsymbol{\psi}^d\cdot\boldsymbol{n}^d)_{|\Gamma}=0\}.
\end{align*}
\end{lemma}
\begin{proof}
For any $\boldsymbol{v}_h\in V_h(g)$, we have $\boldsymbol{v}_h^s=0$ on $\Gamma^s$, and $\boldsymbol{v}_h^d\cdot\boldsymbol{n}^d=0$ on $\Gamma^d$, and $\langle\boldsymbol{v}_h^s\cdot\boldsymbol{n}^s
+\boldsymbol{v}_h^d\cdot\boldsymbol{n}^d,q_h\rangle_e=0$ for any $q_h\in P_{k-1}(e), e\subset\Gamma$. From (\ref{eqn:Pi 2}) and noting that $\boldsymbol{v}_h\cdot\boldsymbol{n}\in P_{k-1}(e), e\subset\partial T$ for any $T\in\mathcal{T}_h$, we get $\Pi_h^s\boldsymbol{v}_h^s\cdot\boldsymbol{n}^s=0$ on $\Gamma^s$, $\Pi_h^d\boldsymbol{v}_h^d\cdot\boldsymbol{n}^d=0$ on $\Gamma^d$, and
$\Pi_h^s\boldsymbol{v}_h^s\cdot\boldsymbol{n}^s
+\Pi_h^d\boldsymbol{v}_h^d\cdot\boldsymbol{n}^d=0$ on $\Gamma$. 

Moreover, from $\nabla\cdot(\Pi_h\boldsymbol{v}_h)_T\in P_{k-1}(T)$ and
$(\nabla\cdot(\Pi_h\boldsymbol{v}_h),1)
=(\Pi_h\boldsymbol{v}_h\cdot\boldsymbol{n},1)_{\partial\Omega}=0$, we have $\nabla\cdot(\Pi_h\boldsymbol{v}_h)\subset Q_h$. And, from the integral by part, (\ref{eqn:Pi 1}), (\ref{eqn:Pi 2}), and $b(\boldsymbol{v}_h,q_h)=0$ for any $q_h\in Q_h$, it holds
\begin{align*}
(\nabla\cdot(\Pi_h\boldsymbol{v}_h),q_h)
=&\sum_{T\in\mathcal{T}_h}(\nabla\cdot(\Pi_h^{i_T}\boldsymbol{v}_h),q_h)_T
=-\sum_{T\in\mathcal{T}_h}(\Pi_h^{i_T}\boldsymbol{v}_h,\nabla q_h)_T+\langle \Pi_h^{i_T}\boldsymbol{v}_h\cdot\boldsymbol{n},q_h\rangle_{\partial T}\\
=&-\sum_{T\in\mathcal{T}_h}(\boldsymbol{v}_h,\nabla q_h)_T+\langle \boldsymbol{v}_h\cdot\boldsymbol{n},q_h\rangle_{\partial T}
=\sum_{T\in\mathcal{T}_h}(\nabla\cdot\boldsymbol{v}_h,q_h)_T=0.
\end{align*}
According to $T\in\Omega^s$ or $\Omega^d$, $i_T$ takes values of $s$ or $d$.The above analysis means $\nabla\cdot(\Pi_h\boldsymbol{v}_h)=0$.
\end{proof}


For the space $V_h(0)$ with norm $\|\cdot\|_X$, we will show three consistency errors caused by the reconstruction operators with dual norm. The first is defined as following
\begin{align}
\label{eqn:consistency error reconstruciton 1}
\|\boldsymbol{\varphi}\circ(1-\Pi_h)\|_{(V_h(0))^*}
=&\sup_{\boldsymbol{\psi}_h\in V_h(0)}\frac{(\boldsymbol{\varphi},(1-\Pi_h)\boldsymbol{\psi}_h)}
{\|\boldsymbol{\psi}_h\|_X},
\end{align}
related to any
$\boldsymbol{\varphi}=(\boldsymbol{\varphi}^s,\boldsymbol{\varphi}^d)\in [L^2(\Omega^s)]^N\times [L^2(\Omega^d)]^N$. Furthermore, if  $\boldsymbol{\varphi}^s\in [H^1(\Omega^s)]^N$, the second consistency error can be shown with the following dual norm
\begin{align}
\label{eqn:consistency error reconstruciton 2}
\|(\boldsymbol{\varphi}^s\cdot\boldsymbol{n}^s)\circ
(1-\Pi_h^s)\|_{(V_h(0))^*}
=\sup_{\boldsymbol{\psi}_h\in V_h(0)}\frac{\langle \boldsymbol{\varphi}^s\cdot\boldsymbol{n}^s,
(1-\Pi_h^s)\boldsymbol{\psi}_h^s\cdot\boldsymbol{n}^s\rangle_{\Gamma}}
{\|\boldsymbol{\psi}_h\|_X}.
\end{align}
Finally, if $\boldsymbol{\varphi}\in V(g)$, the last consistency error will be defined with the dual norm
\begin{align}
\label{eqn:consistency error reconstruciton 3}
\|(1-\Pi_hS_h)\boldsymbol{\varphi}\circ\Pi_h\|_{(V_h(0))^*}
=&\sup_{\boldsymbol{\psi}_h\in V_h(0)}\frac{((1-\Pi_hS_h)\boldsymbol{\varphi},\Pi_h\boldsymbol{\psi}_h)}
{\|\boldsymbol{\psi}_h\|_X}.
\end{align}
\begin{lemma}\label{lem:three consistency errors}
The three consistency errors defined in (\ref{eqn:consistency error reconstruciton 1}),
(\ref{eqn:consistency error reconstruciton 2}), and (\ref{eqn:consistency error reconstruciton 3}) are well-defined, and with the estimations
\begin{align*}
\|\boldsymbol{\varphi}\circ(1-\Pi_h)\|_{(V_h(0))^*}^2
\lesssim &h^2\sum_{T\in\mathcal{T}_h(\Omega^s)}\inf_{\boldsymbol{\varphi}_h^s\in [P_{k-2}(T)]^N}\|\boldsymbol{\varphi}^s-\boldsymbol{\varphi}_h^s\|_T^2,\\
\|(\boldsymbol{\varphi}^s\cdot\boldsymbol{n}^s)\circ
(1-\Pi_h^s)\|_{(V_h(0))^*}^2
\lesssim &h\sum_{e\in\mathcal{E}_h(\Gamma)}\inf_{q_h\in P_{k-1}(e)}
\|\boldsymbol{\varphi}^s\cdot\boldsymbol{n}^s-q_h\|_e^2,\\
\|(1-\Pi_hS_h)\boldsymbol{\varphi}\circ\Pi_h\|_{(V_h(0))^*}^2
\lesssim &\|\boldsymbol{\varphi}-S_h\boldsymbol{\varphi}\|^2
+\|\boldsymbol{\varphi}-\Pi_h\boldsymbol{\varphi}\|^2
+\|(1-\Pi_h)(\boldsymbol{\varphi}-S_h\boldsymbol{\varphi})\|^2.
\end{align*}
\end{lemma}
\begin{proof}
We give the boundedness of the linear functionals on the space $V_h(0)$ in (\ref{eqn:consistency error reconstruciton 1}),
(\ref{eqn:consistency error reconstruciton 2}), and (\ref{eqn:consistency error reconstruciton 3}), respectively.

For the first functional $(\boldsymbol{\varphi},(1-\Pi_h)\boldsymbol{\psi}_h)$, 
from the orthogonality property (\ref{eqn:Pi 1}), Cauchy-Schwarz inequality, estimation (\ref{eqn:Pi inequality}), and the definition of the norm $\|\cdot\|_X$ in (\ref{eqn:consistency error subset}),
it holds
\begin{align*}
(\boldsymbol{\varphi},\boldsymbol{\psi}_h-\Pi_h\boldsymbol{\psi}_h)
=&(\boldsymbol{\varphi}^s,\boldsymbol{\psi}_h^s-\Pi_h^s\boldsymbol{\psi}_h^s)_s
+(\boldsymbol{\varphi}^d,\boldsymbol{\psi}_h^d-\Pi_h^d\boldsymbol{\psi}_h^d)_d
=(\boldsymbol{\varphi}^s,\boldsymbol{\psi}_h^s-\Pi_h^s\boldsymbol{\psi}_h^s)_s\\
\lesssim& h\left(\sum_{T\in\mathcal{T}_h(\Omega^s)}\inf_{\boldsymbol{\varphi}_h^s\in [P_{k-2}(T)]^N}\|\boldsymbol{\varphi}^s-\boldsymbol{\varphi}_h^s\|_T^2\right)^{1/2}
\|\nabla\boldsymbol{\psi}_h^s\|_s\\
\leq& h\left(\sum_{T\in\mathcal{T}_h(\Omega^s)}\inf_{\boldsymbol{\varphi}_h^s\in [P_{k-2}(T)]^N}\|\boldsymbol{\varphi}^s-\boldsymbol{\varphi}_h^s\|_T^2\right)^{1/2}
\|\boldsymbol{\psi}_h\|_X.
\end{align*}

For the second functional $\langle \boldsymbol{\varphi}^s\cdot\boldsymbol{n}^s,
(1-\Pi_h^s)\boldsymbol{\psi}_h^s\cdot\boldsymbol{n}^s\rangle_{\Gamma}$, from the orthogonality property (\ref{eqn:Pi 2}), Cauchy-Schwarz inequality, trace inequality, 
estimation (\ref{eqn:Pi inequality}), and the definition of the norm $\|\cdot\|_X$ in (\ref{eqn:consistency error subset}), we have
\begin{align*}
&\langle \boldsymbol{\varphi}^s\cdot\boldsymbol{n}^s,
(1-\Pi_h^s)\boldsymbol{\psi}_h^s\cdot\boldsymbol{n}^s\rangle_{\Gamma}\\
\leq &
\sum_{e\in\mathcal{E}_h(\Gamma)}\inf_{q_h\in P_{k-1}(e)}
\|\boldsymbol{\varphi}^s\cdot\boldsymbol{n}^s-q_h\|_e
\|(1-\Pi_h^s)\boldsymbol{\psi}_h^s\cdot\boldsymbol{n}^s\|_e\\
\lesssim &\left(\sum_{e\in\mathcal{E}_h(\Gamma)}\inf_{q_h\in P_{k-1}(e)}
\|\boldsymbol{\varphi}^s\cdot\boldsymbol{n}^s-q_h\|_e^2\right)^{1/2}\\
&\left(\sum_{T\in\mathcal{T}_h(\Omega^s)}
h_T^{-1}\|(1-\Pi_h^s)\boldsymbol{\psi}_h^s\|_T^2
+h_T\|\nabla(1-\Pi_h^s)\boldsymbol{\psi}_h^s\|_T^2\right)^{1/2}\\
\lesssim &\left(h\sum_{e\in\mathcal{E}_h(\Gamma)}\inf_{q_h\in P_{k-1}(e)}
\|\boldsymbol{\varphi}^s\cdot\boldsymbol{n}^s-q_h\|_e^2\right)^{1/2}
\|\nabla\boldsymbol{\psi}_h^s\|_s\\
\leq &\left(h\sum_{e\in\mathcal{E}_h(\Gamma)}\inf_{q_h\in P_{k-1}(e)}
\|\boldsymbol{\varphi}^s\cdot\boldsymbol{n}^s-q_h\|_e^2\right)^{1/2}
\|\boldsymbol{\psi}_h\|_X.
\end{align*}

Finally, the Poincar\'e-Friedrichs inequality (Lemma B.66 in \cite{Ern2004}) $\|\boldsymbol{\psi}^s\|\leq C_F\|\nabla\boldsymbol{\psi}^s\|$ also implies the estimate
\begin{align*}
\|\Pi_h^s\boldsymbol{\psi}^s\|\leq \|\boldsymbol{\psi}^s-\Pi_h^s\boldsymbol{\psi}^s\|
+\|\boldsymbol{\psi}^s\|\leq (h+C_F)\|\nabla\boldsymbol{\psi}^s\|
\end{align*}
for any $\boldsymbol{\psi}^s\in V^s$. Noting that $\Pi_h^d\boldsymbol{\psi}_h^d=\boldsymbol{\psi}_h^d$ for any $\boldsymbol{\psi}_h^d\in V_h^d$, we have $\|\Pi_h\boldsymbol{\psi}_h\|\lesssim \|\boldsymbol{\psi}_h\|_X$. Then, for the last functional $((1-\Pi_hS_h)\boldsymbol{\varphi},\Pi_h\boldsymbol{\psi}_h)$, the following estimation holds
\begin{align*}
&((1-\Pi_hS_h)\boldsymbol{\varphi},\Pi_h\boldsymbol{\psi}_h)\\
=&((1-S_h)\boldsymbol{\varphi},\Pi_h\boldsymbol{\psi}_h)
+((1-\Pi_h)\boldsymbol{\varphi},\Pi_h\boldsymbol{\psi}_h)
-((1-\Pi_h)(\boldsymbol{\varphi}-S_h\boldsymbol{\varphi}),\Pi_h\boldsymbol{\psi}_h)\\
\lesssim& (\|\boldsymbol{\varphi}-S_h\boldsymbol{\varphi}\|
+\|\boldsymbol{\varphi}-\Pi_h\boldsymbol{\varphi}\|
+\|(1-\Pi_h)(\boldsymbol{\varphi}-S_h\boldsymbol{\varphi})\|)\|\boldsymbol{\psi}_h\|_X.
\end{align*}
Thus, we complete the proof.
\end{proof}

To obtain a fully pressure-robust discretization, in addition to the Stokes-Darcy equations also the cost functional needs to be modified as follows
\begin{align}
\min_{(\boldsymbol{u}_h,\boldsymbol{\eta}_h)\in V_h(g)\times W
} &\frac{1}{2}\|\Pi_h\boldsymbol{u}_h-\boldsymbol{u}^*\|^2+\frac{\alpha}{2}
\|\boldsymbol{\eta}_h\|^2\\
s.t.&\quad a(\boldsymbol{u}_h,\boldsymbol{\psi}_h)
=(\boldsymbol{f}+\boldsymbol{\eta}_h,\Pi_h\boldsymbol{\psi}_h), \quad \forall\boldsymbol{\psi}_h\in V_h(g).
\end{align}
Here, we have referred to a similar approach to the Stokes optimal control problem \cite{merdonPressureRobustnessContextOptimal2023}.
Again, the optimization problem is equivalent to searching for a solution of the reduced optimality system. Hence, we search $(\boldsymbol{u}_h,\boldsymbol{z}_h,p_h,r_h)\in 
V_h\times V_h\times Q_h\times Q_h$ solving 
\begin{align}
\label{eqn:fully pressure-robust discretization}
\begin{aligned}
 a(\boldsymbol{u}_h,\boldsymbol{\psi}_h)+b(\boldsymbol{\psi}_h,p_h)
&=(f-\alpha^{-1/2}\Pi_h\boldsymbol{z}_h,\Pi_h\boldsymbol{\psi}_h), && \forall
\boldsymbol{\psi}_h\in V_h,\\
b(\boldsymbol{\psi}_h,\phi_h)&=(g,\phi_h), && \forall \phi_h\in Q_h\\
 a(\boldsymbol{\psi}_h,\boldsymbol{z}_h)+b(\boldsymbol{\psi}_h,r_h)
&=\alpha^{-1/2}(\Pi_h\boldsymbol{u}_h-\boldsymbol{u}^d,\Pi_h\boldsymbol{\psi}_h), && \forall
\boldsymbol{\psi}_h\in V_h,\\
b(\boldsymbol{z}_h,\phi_h)&=(g,\phi_h),&& \forall \phi_h\in Q_h.
\end{aligned}
\end{align}

\begin{theorem}\label{the:pressure-robust}
If the exact solutions of (\ref{eqn:PDE problem}) have the regularity
$\boldsymbol{u}^s,\boldsymbol{z}^s\in [H^2(\Omega^s)]^N$, and $p^i,r^i\in H^1(\Omega^i), i=s,d$,
the pressure-robust estimate holds
\begin{align}
\label{eqn:error estmate pressure robust}
\begin{aligned}
\|(S_h\boldsymbol{u}-\boldsymbol{u}_h,S_h\boldsymbol{z}-\boldsymbol{z}_h)\|_X^2
\lesssim &\|(2\nabla\cdot D(\boldsymbol{u}^s),K^{-1}\boldsymbol{u}^d)\circ(1-\Pi_h)\|_{(V_h(0))^*}^2\\
+&\|(2\nabla\cdot D(\boldsymbol{z}^s),K^{-1}\boldsymbol{z}^d)\circ(1-\Pi_h)\|_{(V_h(0))^*}^2\\
+&\|2 D(\boldsymbol{u}^s)\boldsymbol{n}^s\cdot\boldsymbol{n}^s\circ
(1-\Pi_h)\|_{(V_h(0))^*}^2\\
+&\|2 D(\boldsymbol{z}^s)\boldsymbol{n}^s\cdot\boldsymbol{n}^s\circ
(1-\Pi_h)\|_{(V_h(0))^*}^2\\
+&\|(1-\Pi_hS_h)\boldsymbol{u}\circ\Pi_h\|_{(V_h(0))^*}^2\\
+&\|(1-\Pi_hS_h)\boldsymbol{z}\circ\Pi_h\|_{(V_h(0))^*}^2,
\end{aligned}
\end{align}
for the solution $(\boldsymbol{u}_h,\boldsymbol{z}_h)$ of (\ref{eqn:fully pressure-robust discretization}) and the discrete approximation $(S_h\boldsymbol{u},S_h\boldsymbol{z})$ defined in (\ref{eqn:Sh definition}) of the 
exact solutions of (\ref{eqn:PDE problem}). More generally, if the exact solutions of have the regularity
$(\boldsymbol{u}^s,\boldsymbol{u}^d),(\boldsymbol{z}^s,\boldsymbol{z}^d)\in [H^k(\Omega^s)]^N\times [H^{k-1}(\Omega^d)]^N$, $\nabla\cdot \boldsymbol{u}^d, \nabla\cdot \boldsymbol{z}^d\in H^{k-1}(\Omega^d)$, and $p^i,r^i\in H^{k-1}(\Omega^i), i=s,d$ for any $k\geq 2$,
the pressure-robust estimate holds
\begin{align}\label{eqn:error estiamte regular for pressure robust}
\begin{aligned}
\|(\boldsymbol{u}-\boldsymbol{u}_h,\boldsymbol{z}-\boldsymbol{z}_h)\|_X^2
\lesssim h^{2(k-1)}(&\|\boldsymbol{u}^s\|_{k,s}^2+\|\boldsymbol{u}^d\|_{k-1,d}^2+\|\nabla\cdot\boldsymbol{u}^d\|_{k-1,d}^2\\
+&\|\boldsymbol{z}^s\|_{k,s}^2+\|\boldsymbol{z}^d\|_{k-1,d}^2+\|\nabla\cdot\boldsymbol{z}^d\|_{k-1,d}^2).
\end{aligned}
\end{align}
\end{theorem}
\begin{proof}

Since $\boldsymbol{u}^s\in [H^2(\Omega^s)]^N$ and $p^i\in H^1(\Omega^i), i=s,d$, the solutions $(\boldsymbol{u},p)$ of (\ref{eqn:PDE problem}) satisfy (\ref{eqn:state Stokes})$\sim$(\ref{eqn:state boundary 3}) by replacing $\boldsymbol{f}$ with 
$\boldsymbol{f}-\alpha^{-1/2}\boldsymbol{\eta}$. From the definition of the norm $\|\cdot\|_X$ in (\ref{eqn:X norm for u and z}), the property of the projector $S_h$ in (\ref{eqn:Sh definition}), the integration by parts, (\ref{eqn:state Stokes}), (\ref{eqn:state Darcy}), 
the boundary condition (\ref{eqn:state boundary 3}), it reveals
\begin{align}\label{eqn:Shu-uh}
\begin{aligned}
&\|S_h\boldsymbol{u}-\boldsymbol{u}_h\|_X^2\\
\lesssim&a(\boldsymbol{u}-\boldsymbol{u}_h,S_h\boldsymbol{u}-\boldsymbol{u}_h)\\
=&-2\mu(\nabla\cdot D(\boldsymbol{u}^s),(S_h\boldsymbol{u}-\boldsymbol{u}_h)^s)_s
+\mu(K^{-1}\boldsymbol{u}^d,(S_h\boldsymbol{u}-\boldsymbol{u}_h)^d)_d\\
&+2\mu\langle D(\boldsymbol{u}^s)\boldsymbol{n}^s\cdot\boldsymbol{n}^s,
(S_h\boldsymbol{u}-\boldsymbol{u}_h)^s\cdot\boldsymbol{n}^s\rangle_{\Gamma}
-(\boldsymbol{f}-\alpha^{-1/2}\Pi_h\boldsymbol{z}_h,
\Pi_h(S_h\boldsymbol{u}-\boldsymbol{u}_h))\\
=&-2\mu(\nabla\cdot D(\boldsymbol{u}^s),(1-\Pi_h^s)(S_h\boldsymbol{u}-\boldsymbol{u}_h)^s)_s
+\mu(K^{-1}\boldsymbol{u}^d,(1-\Pi_h^d)(S_h\boldsymbol{u}-\boldsymbol{u}_h)^d)_d\\
&
-\alpha^{-1/2}(\boldsymbol{z}-\Pi_h\boldsymbol{z}_h,
\Pi_h(S_h\boldsymbol{u}-\boldsymbol{u}_h))
+2\mu\langle D(\boldsymbol{u}^s)\boldsymbol{n}^s\cdot\boldsymbol{n}^s,
(S_h\boldsymbol{u}-\boldsymbol{u}_h)^s\cdot\boldsymbol{n}^s\rangle_{\Gamma}\\
&-(\nabla p^s,\Pi_h^s(S_h\boldsymbol{u}-\boldsymbol{u}_h)^s)_s
-(\nabla p^d,\Pi_h^d(S_h\boldsymbol{u}-\boldsymbol{u}_h)^d)_d.
\end{aligned}
\end{align}
For the last two terms in (\ref{eqn:Shu-uh}), 
using the integration by parts, the boundary condition (\ref{eqn:state boundary 2}), and Lemma~\ref{lem:Pi}, we have
\begin{align}\label{eqn:pressure related terms}
\begin{aligned}
&-(\nabla p^s,\Pi_h^s(S_h\boldsymbol{u}-\boldsymbol{u}_h)^s)_s
-(\nabla p^d,\Pi_h^d(S_h\boldsymbol{u}-\boldsymbol{u}_h)^d)_d\\
=& (p^s,\nabla\cdot\Pi_h^s(S_h\boldsymbol{u}-\boldsymbol{u}_h)^s)_s
+(p^d,\nabla\cdot\Pi_h^d(S_h\boldsymbol{u}-\boldsymbol{u}_h)^d)_d\\
&-\langle p^s,\Pi_h^s(S_h\boldsymbol{u}-\boldsymbol{u}_h)^s\cdot\boldsymbol{n}^s\rangle_{\Gamma^s}
-\langle p^d,\Pi_h^d(S_h\boldsymbol{u}-\boldsymbol{u}_h)^d\cdot\boldsymbol{n}^d\rangle_{\Gamma^d}\\
&-\langle p^s,\Pi_h^s(S_h\boldsymbol{u}-\boldsymbol{u}_h)^s\cdot\boldsymbol{n}^s\rangle_{\Gamma}
-\langle p^d,\Pi_h^d(S_h\boldsymbol{u}-\boldsymbol{u}_h)^d\cdot\boldsymbol{n}^d\rangle_{\Gamma}\\
=&-\langle p^s,\Pi_h^s(S_h\boldsymbol{u}-\boldsymbol{u}_h)^s\cdot\boldsymbol{n}^s\rangle_{\Gamma}
-\langle p^d,\Pi_h^d(S_h\boldsymbol{u}-\boldsymbol{u}_h)^d\cdot\boldsymbol{n}^d\rangle_{\Gamma}\\
=&-2\mu\langle D(\boldsymbol{u}^s)\boldsymbol{n}^s\cdot\boldsymbol{n}^s,
\Pi_h^s(S_h\boldsymbol{u}-\boldsymbol{u}_h)^s\cdot\boldsymbol{n}^s\rangle_{\Gamma}\\
&-\langle p^d,\Pi_h^s(S_h\boldsymbol{u}
-\boldsymbol{u}_h)^s\cdot\boldsymbol{n}^s
+\Pi_h^d(S_h\boldsymbol{u}-\boldsymbol{u}_h)^d\cdot\boldsymbol{n}^d\rangle_{\Gamma}\\
=&-2\mu\langle D(\boldsymbol{u}^s)\boldsymbol{n}^s\cdot\boldsymbol{n}^s,
\Pi_h^s(S_h\boldsymbol{u}-\boldsymbol{u}_h)^s\cdot\boldsymbol{n}^s\rangle_{\Gamma}.
\end{aligned}
\end{align}
Combining (\ref{eqn:Shu-uh}) and (\ref{eqn:pressure related terms}), we have
\begin{align}
\label{eqn:Su uh X robust}
\begin{aligned}
\|S_h\boldsymbol{u}-\boldsymbol{u}_h\|_X^2\lesssim
&((-2\mu\nabla\cdot D(\boldsymbol{u}^s),K^{-1}\boldsymbol{u}^d),
(1-\Pi_h)(S_h\boldsymbol{u}-\boldsymbol{u}_h))\\
&-\alpha^{-1/2}(\boldsymbol{z}-\Pi_h\boldsymbol{z}_h,
\Pi_h(S_h\boldsymbol{u}-\boldsymbol{u}_h))\\
&+2\mu\langle D(\boldsymbol{u}^s)\boldsymbol{n}^s\cdot\boldsymbol{n}^s,
(1-\Pi_h^s)(S_h\boldsymbol{u}-\boldsymbol{u}_h)^s\cdot\boldsymbol{n}^s\rangle_{\Gamma}.\\
=&I_1+I_2+I_3
\end{aligned}
\end{align}

Analogously, one obtains
\begin{align}
\label{eqn:Sz zh X robust}
\begin{aligned}
\|S_h\boldsymbol{z}-\boldsymbol{z}_h\|_X^2
\lesssim&((-2\mu\nabla\cdot D(\boldsymbol{z}^s),K^{-1}\boldsymbol{z}^d),
(1-\Pi_h)(S_h\boldsymbol{z}-\boldsymbol{z}_h))\\
&+\alpha^{-1/2}(\boldsymbol{u}-\Pi_h\boldsymbol{u}_h,
\Pi_h(S_h\boldsymbol{z}-\boldsymbol{z}_h))\\
&+2\mu\langle D(\boldsymbol{z}^s)\boldsymbol{n}^s\cdot\boldsymbol{n}^s,
(1-\Pi_h^s)(S_h\boldsymbol{z}-\boldsymbol{z}_h)^s\cdot\boldsymbol{n}^s\rangle_{\Gamma}\\
=&I_4+I_5+I_6
\end{aligned}
\end{align}
Next, noting $S_h\boldsymbol{u}-\boldsymbol{u}_h, S_h\boldsymbol{z}-\boldsymbol{z}_h\in V_h(0)$, we will estimate the six terms in (\ref{eqn:Su uh X robust}) and (\ref{eqn:Sz zh X robust}). For the terms $I_1,I_4,I_3$, and $I_6$, we have
\begin{align*}
I_1\leq& 
\|(2\mu\nabla\cdot D(\boldsymbol{u}^s),K^{-1}\boldsymbol{u}^d)\circ(1-\Pi_h)\|_{(V_h(0))^*}
\|S_h\boldsymbol{u}-\boldsymbol{u}_h\|_X\\
I_4\leq&
\|(2\mu\nabla\cdot D(\boldsymbol{z}^s),K^{-1}\boldsymbol{z}^d)\circ(1-\Pi_h)\|_{(V_h(0))^*}
\|S_h\boldsymbol{z}-\boldsymbol{z}_h\|_X\\
I_3\leq &\|2\mu D(\boldsymbol{u}^s)\boldsymbol{n}^s\cdot\boldsymbol{n}^s\circ
(1-\Pi_h)\|_{(V_h(0))^*}
\|S_h\boldsymbol{u}-\boldsymbol{u}_h\|_X\\
I_6\leq& \|2\mu D(\boldsymbol{z}^s)\boldsymbol{n}^s\cdot\boldsymbol{n}^s\circ
(1-\Pi_h)\|_{(V_h(0))^*}
\|S_h\boldsymbol{z}-\boldsymbol{z}_h\|_X.
\end{align*}
Then, for the term $I_2+I_4$, it can be obtained
\begin{align*}
&I_2+I_5\\
=&-\alpha^{-1/2}(\boldsymbol{z}-\Pi_hS_h\boldsymbol{z},
\Pi_h(S_h\boldsymbol{u}-\boldsymbol{u}_h))
-\alpha^{-1/2}(\Pi_h(S_h\boldsymbol{z}-\boldsymbol{z}_h),
\Pi_h(S_h\boldsymbol{u}-\boldsymbol{u}_h))\\
&+\alpha^{-1/2}(\boldsymbol{u}-\Pi_hS_h\boldsymbol{u},
\Pi_h(S_h\boldsymbol{z}-\boldsymbol{z}_h))
+\alpha^{-1/2}(\Pi_h(S_h\boldsymbol{u}-\boldsymbol{u}_h),
\Pi_h(S_h\boldsymbol{z}-\boldsymbol{z}_h))\\
\leq& \alpha^{-1/2}(\|(1-\Pi_hS_h)\boldsymbol{z}\circ\Pi_h\|_{(V_h(0))^*}
\|S_h\boldsymbol{u}-\boldsymbol{u}_h\|_X\\
&+\|(1-\Pi_hS_h)\boldsymbol{u}\circ\Pi_h\|_{(V_h(0))^*}
\|S_h\boldsymbol{z}-\boldsymbol{z}_h\|_X).
\end{align*}
The estimates of $I_1\sim I_6$ yield (\ref{eqn:error estmate pressure robust}).

Next, for $\boldsymbol{v}=\boldsymbol{u}$ or $\boldsymbol{z}$, we will prove (\ref{eqn:error estiamte regular for pressure robust}) under the regularity $(\boldsymbol{v}^s,\boldsymbol{v}^d)\in [H^k(\Omega^s)]^N\times [H^{k-1}(\Omega^d)]^N$ and $\nabla\cdot \boldsymbol{v}^d\in H^{k-1}(\Omega^d)$ for any $k\geq 2$ .

Firstly, for the first two terms in (\ref{eqn:error estmate pressure robust}), from Lemma~\ref{lem:three consistency errors} and projection property, we have
\begin{align}\label{eqn:first two term estimate}
\begin{aligned}
&\|(2\nabla\cdot D(\boldsymbol{v}^s),K^{-1}\boldsymbol{v}^d)\circ(1-\Pi_h)\|_{(V_h(0))^*}^2\\
\lesssim &h^2\sum_{T\in\mathcal{T}_h(\Omega^s)}\inf_{\boldsymbol{\varphi}_h^s\in [P_{k-2}(T)]^N}\|\nabla\cdot D(\boldsymbol{v}^s)-\boldsymbol{\varphi}_h^s\|_T^2\\
\lesssim &h^{2(k-1)}\|\boldsymbol{v}^s\|_{k,s}^2.
\end{aligned}
\end{align}

For the third and fourth terms in (\ref{eqn:error estmate pressure robust}), let $\boldsymbol{\Psi}_h^s$ to be the local $L^2$ projection of $D(\boldsymbol{v}^s)$ into $\{\boldsymbol{\Phi}_h\in [L^2(\Omega^s)]^{N\times N}~|~\boldsymbol{\Phi}_{h|T}\in [P_{k-1}(T)]^{N\times N},~T\in \mathcal{T}_h(\Omega^s)\}$. From Lemma~\ref{lem:three consistency errors}, trace inequality, and projection property, it is obtained
\begin{align}\label{eqn:third and fourth term estimate}
\begin{aligned}
&\|2 D(\boldsymbol{v}^s)\boldsymbol{n}^s\cdot\boldsymbol{n}^s\circ
(1-\Pi_h)\|_{(V_h(0))^*}^2\\
\lesssim &h\sum_{e\in\mathcal{E}_h(\Gamma)}\inf_{q_h\in P_{k-1}(e)}
\|D(\boldsymbol{v}^s)\boldsymbol{n}^s\cdot\boldsymbol{n}^s-q_h\|_e^2\\
\lesssim &h\sum_{e\in\mathcal{E}_h(\Gamma)}
\|D(\boldsymbol{v}^s)\boldsymbol{n}^s\cdot\boldsymbol{n}^s-\boldsymbol{\Psi}_h^s\boldsymbol{n}^s\cdot\boldsymbol{n}^s\|_e^2\\
\lesssim &h\sum_{T\in\mathcal{T}_h(\Gamma)}
h_T^{-1}\|D(\boldsymbol{v}^s)-\boldsymbol{\Psi}_h^s\|_T^2+h_T|D(\boldsymbol{v}^s)-\boldsymbol{\Psi}_h^s|_{1,T}^2\\
\lesssim &h^{2(k-1)}\|\boldsymbol{v}^s\|_{k,s}^2,
\end{aligned}
\end{align}
where $\mathcal{T}_h(\Gamma)$ is defined in ($\ref{eqn:psi_s estimate}$). 

For the last two terms in (\ref{eqn:error estmate pressure robust}), 
from Lemma~\ref{lem:three consistency errors}, 
the following estimate holds
\begin{align}\label{eqn:last two term}
\|(1-\Pi_hS_h)\boldsymbol{v}\circ\Pi_h\|_{(V_h(0))^*}^2
\lesssim \|\boldsymbol{v}-S_h\boldsymbol{v}\|^2
+\|\boldsymbol{v}-\Pi_h\boldsymbol{v}\|^2
+\|(1-\Pi_h)(\boldsymbol{v}-S_h\boldsymbol{v})\|^2.
\end{align}
From the Poincar\'e-Friedrichs inequality (Lemma B.66 in \cite{Ern2004}) and (\ref{eqn:approximation error u and z}), we get
\begin{align}\label{eqn:estiate Sh}
\|\boldsymbol{v}-S_h\boldsymbol{v}\|^2\lesssim \|\boldsymbol{v}-S_h\boldsymbol{v}\|_X^2\lesssim & h^{2(k-1)}(\|\boldsymbol{v}^s\|_{k,s}^2+\|\boldsymbol{v}^d\|_{k-1,d}^2+\|\nabla\cdot \boldsymbol{v}^d\|_{k-1,d}^2).
\end{align}
Using inequality (\ref{eqn:Pi inequality}) in $\Omega^s$ and $\Omega^d$, respectively, it is clear that
\begin{align}\label{eqn:estiate Pih}
\|\boldsymbol{v}-\Pi_h\boldsymbol{v}\|^2\lesssim & \|\boldsymbol{v}^s-\Pi_h^s\boldsymbol{v}^s\|_{0,s}^2+\|\boldsymbol{v}^d-\Pi_h^d\boldsymbol{v}^d\|_{0,d}^2
\lesssim h^{2k}\|\boldsymbol{v}^s\|_{k,s}^2+h^{2(k-1)}\|\boldsymbol{v}^d\|_{k-1,d}^2.
\end{align}
Noting $\Pi_h^d(S_h\boldsymbol{v})^d=(S_h\boldsymbol{v})^d$ and combining (\ref{eqn:estiate Sh}) and (\ref{eqn:estiate Pih}), it holds
\begin{align}\label{eqn:estiate Sh and Phi}
\begin{aligned}
\|(1-\Pi_h)(\boldsymbol{v}-S_h\boldsymbol{v})\|^2\lesssim &\|(1-\Pi_h^s)(\boldsymbol{v}-S_h\boldsymbol{v})^s\|_s^2+\|(1-\Pi_h^d)(\boldsymbol{v}-S_h\boldsymbol{v})^d\|_d^2\\
\lesssim &\|(1-\Pi_h^s)(\boldsymbol{v}-S_h\boldsymbol{v})^s\|_s^2+\|(1-\Pi_h^d)\boldsymbol{v}^d\|_d^2\\
\lesssim &h^{2}\|(\boldsymbol{v}-S_h\boldsymbol{v})^s\|_{1,s}^2+h^{2(k-1)}\|\boldsymbol{v}^d\|_{k-1,d}^2\\
\lesssim &h^{2}\|\boldsymbol{v}-S_h\boldsymbol{v}\|_X^2+h^{2(k-1)}\|\boldsymbol{v}^d\|_{k-1,d}^2\\
\lesssim &h^{2k}(\|\boldsymbol{v}^s\|_{k,s}^2+\|\nabla\cdot \boldsymbol{v}^d\|_{k-1,d}^2)+h^{2(k-1)}\|\boldsymbol{v}^d\|_{k-1,d}^2.
\end{aligned}
\end{align}
Then, from (\ref{eqn:last two term}), (\ref{eqn:estiate Sh}), (\ref{eqn:estiate Pih}), and (\ref{eqn:estiate Sh and Phi}), we get the estimate of the last two terms in (\ref{eqn:error estmate pressure robust})
\begin{align}\label{eqn:last two term estimate}
\|(1-\Pi_hS_h)\boldsymbol{v}\circ\Pi_h\|_{(V_h(0))^*}^2
\lesssim  h^{2(k-1)}(\|\boldsymbol{v}^s\|_{k,s}^2+\|\boldsymbol{v}^d\|_{k-1,d}^2+\|\nabla\cdot \boldsymbol{v}^d\|_{k-1,d}^2).
\end{align}

Finally, from (\ref{eqn:approximation error u and z}), (\ref{eqn:error estmate pressure robust}), (\ref{eqn:last two term estimate}), (\ref{eqn:third and fourth term estimate}), and (\ref{eqn:error estmate pressure robust}), the proof is completed.
\end{proof}
\begin{remark}
Employing a divergence-free reconstruction operator, the conditions of exact divergence-freedom and interface normal continuity are precisely satisfied, as elucidated in Lemma~\ref{lem:Pi}. Additionally, within the realm of error estimation, terms correlated with pressure are observed to vanish, a phenomenon delineated by the equation (\ref{eqn:pressure related terms}).
\end{remark}

\section*{Acknowledgments}
The work of J. Li is supported by the National Natural Science Foundation of China (grant No. 12201310) and the Natural Science Foundation of Jiangsu Province (grant No. BK20190766). The work of J. Zhang is supported by the National Natural Science Foundation of China (grant No. 12301469) and the Natural Science Foundation of Jiangsu Province (grant No. BK20210540). 
The work of R. Zhang was supported by the National Natural Science Foundation of China (grant No. 22341302), the National Key Research and Development Program of China (grant No. 2020YFA0713602, 2023YFA1008803), and the Key Laboratory of Symbolic Computation and Knowledge Engineering of Ministry of Education of China housed at Jilin University.


\bibliographystyle{plain}

\begin{thebibliography}{10}

\bibitem{Boffi2013}
Daniele Boffi, Franco Brezzi, and Michel Fortin.
\newblock {\em {Mixed Finite Element Methods and Applications}}.
\newblock Springer Berlin, Heidelberg, 2013.

\bibitem{Brezzi1991}
Franco Brezzi and Michel Fortin.
\newblock {\em {Mixed and Hybrid Finite Element Methods}}, volume~15 of {\em
  Springer Series in Computational Mathematics}.
\newblock Springer New York, 1991.

\bibitem{cockburnNoteDiscontinuousGalerkin2007}
Bernardo Cockburn, Guido Kanschat, and Dominik Sch\"{o}tzau.
\newblock A note on discontinuous galerkin divergence-free solutions of the
  navier stokes equations.
\newblock {\em Journal of Scientific Computing}, 31:61--73, 2007.

\bibitem{Ern2004}
Alexandre Ern and Jean-Luc Guermond.
\newblock {\em {Theory and Practice of Finite Elements}}.
\newblock Springer New York, 2004.

\bibitem{falkStokesComplexesConstruction2013}
Richard~S. Falk and Michael Neilan.
\newblock Stokes {Complexes} and the {Construction} of {Stable} {Finite}
  {Elements} with {Pointwise} {Mass} {Conservation}.
\newblock {\em SIAM Journal on Numerical Analysis}, 51:1308--1326, 2013.

\bibitem{Gaspoz2020}
Fernando Gaspoz, Christian Kreuzer, Andreas Veeser, and Winnifried Wollner.
\newblock {Quasi-best approximation in optimization with PDE constraints}.
\newblock {\em Inverse Problems}, 36(1):1--29, 2020.

\bibitem{guzmanConformingDivergencefreeStokes2014}
Johnny Guzm\'an and Michael Neilan.
\newblock Conforming and divergence-free {Stokes} elements on general
  triangular meshes.
\newblock {\em Mathematics of Computation}, 83:15--36, 2014.

\bibitem{Hinze2005}
M.~Hinze.
\newblock {A variational discretization concept in control constrained
  optimization: The linear-quadratic case}.
\newblock {\em Computational Optimization and Applications}, 30(1):45--61,
  2005.

\bibitem{johnDivergenceConstraintMixed2017}
Volker John, Alexander Linke, Christian Merdon, Michael Neilan, and Leo~G.
  Rebholz.
\newblock On the {Divergence} {Constraint} in {Mixed} {Finite} {Element}
  {Methods} for {Incompressible} {Flows}.
\newblock {\em SIAM Review}, 59:492--544, 2017.

\bibitem{Layton2002}
William~J. Layton, Friedhelm Schieweck, and Ivan Yotov.
\newblock {Coupling fluid flow with porous media flow}.
\newblock {\em SIAM Journal on Numerical Analysis}, 40:2195--2218, 2002.

\bibitem{ledererDivergencefreeReconstructionOperators2017}
Philip~L. Lederer, Alexander Linke, Christian Merdon, and Joachim Sch?berl.
\newblock Divergence-free {Reconstruction} {Operators} for {Pressure}-{Robust}
  {Stokes} {Discretizations} with {Continuous} {Pressure} {Finite} {Elements}.
\newblock {\em SIAM Journal on Numerical Analysis}, 55:1291--1314, 2017.

\bibitem{Linke2016}
Alexander Linke, Gunar Matthies, and Lutz Tobiska.
\newblock {Robust Arbitrary Order Mixed Finite Element Methods for the
  Incompressible Stokes Equations with pressure independent velocity errors}.
\newblock {\em ESAIM: Mathematical Modelling and Numerical Analysis},
  50(1):289--309, 2016.

\bibitem{Reyes2015}
Juan Carlos~De los Reyes.
\newblock {\em {Numerical PDE-Constrained Optimization}}.
\newblock Springer, 2015.

\bibitem{merdonPressureRobustnessContextOptimal2023}
Christian Merdon and Winnifried Wollner.
\newblock Pressure-{Robustness} in the {Context} of {Optimal} {Control}.
\newblock {\em SIAM Journal on Control and Optimization}, 61:342--360, 2023.

\bibitem{Raviart1977}
P.~A. Raviart and J.~M. Thomas.
\newblock {\em A mixed finite element method for 2-nd order elliptic problems.
  Mathematical Aspects of Finite Element Methods.Lecture Notes in
  Mathematics,vol 606}.
\newblock Springer Berlin Heidelberg, 1977,292--315.

\bibitem{wangNewFiniteElement2007}
Junping Wang and Xiu Ye.
\newblock New {Finite} {Element} {Methods} in {Computational} {Fluid}
  {Dynamics} by {H}(div) {Elements}.
\newblock {\em SIAM Journal on Numerical Analysis}, 45:1269--1286, 2007.

\end{thebibliography}

\end{document}